\documentclass[11pt]{amsart}
\usepackage[english]{babel}
\usepackage[latin1,utf8]{inputenc}
\usepackage{amssymb,geometry,color,xcolor,graphicx}
\usepackage{amsmath,amsthm}
\usepackage{booktabs}
\usepackage{caption}
\usepackage{datetime}
\usepackage{dsfont}
\usepackage{amstext}
\usepackage{graphicx}
\usepackage{fancyhdr}
\usepackage{latexsym}
\usepackage{mathrsfs}
\usepackage{cases}
\usepackage{mathtools}
\usepackage{bm}
\usepackage{subcaption}
\usepackage{wrapfig}
\usepackage{tikz}\usetikzlibrary{positioning, shapes.geometric, shapes, arrows.meta, calc}
\usepackage{pgfplots}
\usepackage{ifthen}
\pgfplotsset{width=10cm,compat=1.9}

\usepackage{subcaption}

\usepackage{graphicx}
\usepackage{xcolor,listings}
\usepackage[numbered]{matlab-prettifier}
\usepackage{hyperref}
\hypersetup{
    colorlinks=true, 
    linktoc=all,     
    linkcolor=blue,  
}
\usepackage{listofitems} 
\usetikzlibrary{arrows.meta} 
\usepackage[outline]{contour} 
\contourlength{1.4pt}

\tikzset{>=latex} 

\usepackage{xcolor}
\colorlet{myred}{red!80!black}
\colorlet{myblue}{blue!80!black}
\colorlet{mygreen}{green!60!black}
\colorlet{myorange}{orange!70!red!60!black}
\colorlet{mydarkred}{red!30!black}
\colorlet{mydarkblue}{blue!40!black}
\colorlet{mydarkgreen}{green!30!black}
\colorlet{mylime}{lime!80!black}

\tikzset{
  >=latex, 
  node/.style={thick,circle,draw=myblue,minimum size=22,inner sep=0.5,outer sep=0.6},
  node in/.style={node,green!20!black,draw=mygreen!30!black,fill=mygreen!25},
  node hidden/.style={node,blue!20!black,draw=myblue!30!black,fill=myblue!20},
  node hiddenrelu/.style={node,orange!20!black,draw=myorange!30!black,fill=myorange!20},
  node hiddensigmoid/.style={node,lime!20!black,draw=mylime!30!black,fill=mylime!20},
  node out/.style={node,red!20!black,draw=myred!30!black,fill=myred!20},
  connect/.style={thick,mydarkblue}, 
  connect arrow/.style={-{Latex[length=4,width=3.5]},thick,mydarkblue,shorten <=0.5,shorten >=1},
  node 1/.style={node in}, 
  node 2/.style={node hidden},
  node 3/.style={node hidden},
  node 4/.style={node hiddensigmoid},
  node 5/.style={node out},
}
\def\nstyle{int(\lay<\Nnodlen?min(3,\lay):5)} 
\def\n2style{int(\lay<\Nnodlen+4?min(4,\lay-1):5)} 

\numberwithin{equation}{section}

\newcommand{\udef}{\mathrel{\mathop:}=}

\newcommand{\R}{\mathbb{R}}

\newcommand{\N}{\mathbb{N}}

\newcommand{\de}{\,\mathrm{d}}
\usepackage{dsfont}

\newcommand{\sgn}{\mathrm{sgn}}
\newcommand{\D}{\mathcal{D}}
\newcommand{\Phiin}{\Phi_i(\theta^{(n)})}

\theoremstyle{plain}
\newtheorem{thm}{Theorem}[section]
\newtheorem{cor}[thm]{Corollary}
\newtheorem{prop}[thm]{Proposition}
\newtheorem{lem}[thm]{Lemma}

\newtheorem{defn}[thm]{Definition}
\newtheorem{rem}[thm]{Remark}
\newtheorem{exm}[thm]{Example}

\newcommand{\bmat}[1]{ \begin{bmatrix}#1\end{bmatrix}}
\let\tilde\widetilde
\let\hat\widehat

\newcommand{\parder}[2]{\frac{\partial#1}{\partial#2}}

\newcommand{\muPL}{\mu-\textrm{PL}^*}

\renewcommand{\ss}{\scriptstyle}
\def\svdots{\vbox{\baselineskip=1.5pt\lineskiplimit=0pt
	\kern1.5pt \hbox{$\ss .$}\hbox{$\ss .$}\hbox{$\ss .$}}}

\begin{document}

\title{PHYSICS INFORMED NEURAL NETWORKS FOR LEARNING THE HORIZON SIZE IN BOND-BASED PERIDYNAMIC MODELS}
\author[Difonzo]{Fabio V. Difonzo}
\address{Istituto per le Applicazioni del Calcolo \textquotedblleft Mauro Picone\textquotedblright, Consiglio Nazionale delle Ricerche, Via G. Amendola 122/I, 70126 Bari, Italy}
\email{fabiovito.difonzo@cnr.it}
\address{Departement of Engineering, LUM University Giuseppe Degennaro, S.S. 100 km 18, 70010 Casamassima (BA), Italy}
\email{difonzo@lum.it}
\author[Lopez]{Luciano Lopez}
\address{Dipartimento di Matematica, Universit\`a degli Studi di Bari Aldo Moro, Via E. Orabona 4, 70125 Bari, Italy}
\email{luciano.lopez@uniba.it}
\author[Pellegrino]{Sabrina F. Pellegrino}
\address{Dipartimento di Ingegneria Elettrica e dell'Informazione, Politecnico di Bari, Via E. Orabona 4, 70125 Bari, Italy}
\email{sabrinafrancesca.pellegrino@poliba.it}

\subjclass{34A36, 15B99}

\keywords{Physics Informed Neural Network, Bond-Based Peridynamic Theory, Horizon}

\null\hfill Version of \today $, \,\,\,$ \xxivtime

\begin{abstract}
This paper broaches the peridynamic inverse problem of determining the horizon size of the kernel function in a one-dimensional model of a linear microelastic material. We explore different kernel functions, including V-shaped, distributed, and tent kernels. The paper presents numerical experiments using PINNs to learn the horizon parameter for problems in one and two spatial dimensions. The results demonstrate the effectiveness of PINNs in solving the peridynamic inverse problem, even in the presence of challenging kernel functions. We observe and prove a one-sided convergence behavior of the Stochastic Gradient Descent method towards a global minimum of the loss function, suggesting that the true value of the horizon parameter is an unstable equilibrium point for the PINN's gradient flow dynamics.
\end{abstract}

\maketitle

\pagestyle{myheadings}
\thispagestyle{plain}
\markboth{F.V. DIFONZO, L. LOPEZ AND S.F. PELLEGRINO}{PINNs FOR THE LEARNING THE HORIZON SIZE IN BOND-BASED PERIDYNAMIC MODELS}

\section{Introduction to the peridynamic inverse problem}

Peridynamics is an alternative theory of solid mechanics introduced by Silling in~\cite{SILLING2000} with the aim to reformulate the basic mathematical description of the motion of a continuum in such a way that the identical equations hold either on or off of a jump discontinuity such as a crack. The theory was developed to answer several engineering problems such as the monitoring of the structural damage of an aircraft components and several benchmark engineering problems can be found in literature, see for instance~\cite{MadenciBook2014}.

The theory accounts for the nonlocal interactions among particles located within a region of finite distance, whose size is parametrized by a positive constant value $\delta$. This length-parameter is related to the characteristic length-scale of the material under consideration. Damage is incorporated in the theory at the level of these interactions by particles, so fractures occur as a natural outgrowth of the equation of motion.
In the bond-based peridynamic formulation, the nonlocal interaction between two material particles is called bond and is modeled as a spring between the two points. This represents the main fundamental difference between peridynamics and classical theory, where interactions occur only in presence of direct contact forces.

From a mathematical point of view, partial derivatives are replaced by an integral operator such that the acceleration of any particle $x$ in the reference configuration at any time $t$ is given by
\begin{equation}
\label{eq:bondperid}
\parder{^2u}{t^2}(x,t)=\int_{B_{\delta}(x)}f\left(u(y,t)-u(x,t),y-x\right)\,\de y,
\end{equation}
where $u$ is the displacement field and $f$ is a pairwise force function whose value is the force per unit volume squared that the particle $y$ exerts on the particle $x$. If we consider microelastic materials, we can assume that the pairwise force function $f$ takes the form
\begin{equation}
\label{eq:f}
f\left(u(y,t)-u(x,t),y-x\right) = C(|x-y|) \left(u(x,t)-u(y,t)\right),
\end{equation}
where $C$ is the material's micromodulus function representing the kernel function governing the interaction's strength.

In this paper, we consider the one-dimensional case model of the dynamic response of an infinite bar composed of a linear microelastic material, described by the following PDE in peridynamic formulation:
\begin{equation}\label{eq:periPDE}
\parder{^2u}{t^2}(x,t)=\int_{\R}C(|x-y|)[u(x,t)-u(y,t)]\,\de y,
\end{equation}
where $C:\R\to\R$ represents the so-called kernel function. We further guarantee the consistency with Newton's third law by requiring that $C$ be nonnegative and even:
\[
C\left(\xi\right)=C\left(-\xi\right),\quad \xi\in\R.
\]

As a result of the assumption of long-range interactions, the motion is dispersive and by examining the steady propagation of sinusoidal waves characterized by an angular frequency $\omega$, a wave number $k$ and a phase speed $c=\frac{\omega}{k}$, we find the following dispersive relation
\begin{equation}
\label{eq:dispersive}
\omega=\omega(k)=\sqrt{M(k)},\quad\text{where }M(k)\udef\int_{\R}\left(1-\cos(k\xi) C(\xi)\right) \de\xi.
\end{equation}
Additionally, it is reasonable assume that interactions between two material particles becomes negligible as the distance among them becomes large. Thus, we have
\begin{equation}
\label{eq:Clim}
\lim_{\xi\to\pm\infty} C(\xi) = 0.
\end{equation}

If a material is characterized by a finite horizon, so that no interactions happen within particles that have relative distance greater than $\delta$, then we can assume that the support of the kernel function is given by $[-\delta,\delta]$ and in this case equation~\eqref{eq:Clim} is automatically satisfied. Moreover, under such assumption, the model~\eqref{eq:periPDE} writes as
\begin{equation}
\label{eq:periPDEsupport}
\parder{^2u}{t^2}(x,t)=\int_{B_{\delta}(x)}C(|x-y|)[u(x,t)-u(y,t)]\,\de y.
\end{equation}
From a physical point of view, the function $C$ characterizes the stiffness of a material in presence of long-range forces and involves a length-scale parameter $\delta$ which represents a measure of the nonlocality degree of the model able to capture of the dispersive effects of the long-range interactions. We can, thus, assume that for linear microelastic material
\[
C = C\left(|x-x'|;\delta\right).
\]
In the limit case of short-range interactions, namely in the case $\delta\to 0$, the peridynamic theory converges to the classic elasticity theory, see~\cite{WECKNER2005}. Hereafter, $C$ will be always assumed to be compactly supported.

We augment equation~\eqref{eq:periPDEsupport} by two initial conditions
\begin{equation}
\label{eq:IC}
u(x,0)=u_0(x),\qquad \frac{\partial u}{\partial t}(x,0) = v_0(x), \qquad x\in\Omega,
\end{equation}
then the initial-value problem~\eqref{eq:periPDEsupport}-\eqref{eq:IC} is well-posed (see~\cite{Emmrich_Puhst_2015}) with possible dispersive behaviors of the solution as a consequence of long-range forces in the following functional space.

Let $X=\mathcal{C}_b^1\left(\Omega\right)$ be the space of bounded continuous and differentiable functions or $X=W^{1,p}\left(\Omega\right)$, with $1\le p \le \infty$, then the following Theorem holds.
\begin{thm}[see~\cite{Emmrich_Puhst_2015}]
\label{th:wellposedness}
Let the initial data in~\eqref{eq:IC} be given in $X$ and assume $C\in L^1(\R)$. Then the initial-value problem associated with~\eqref{eq:periPDEsupport} is locally well-posed with solution in $\mathcal{C}^2(X;[0,T])$, for any $T>0$.
\end{thm}

It is clear that a different microelastic material corresponds to a different kernel function and, as a consequence, the kernel function involved in the model provides different constitutive models.

Among the numerous proposals of kernel functions in literature of peridynamic theory, according to~\cite{WECKNER2005} we will particularly draw our attention on Gauss-type kernels of the form
\begin{equation}\label{eq:gaussKernel}
C(\xi)=\lambda e^{-\mu \xi^2},\qquad \lambda,\,\mu>0,
\end{equation}
or on V-shaped kernels of the type
\begin{equation}\label{eq:VKernel}
C(\xi)=
\begin{cases}
    \lambda|\xi|,\quad&|\xi|\leq\delta,\\
    0,\quad&|\xi|>\delta,
\end{cases}\qquad \lambda>0.
\end{equation}
Moreover, we will consider a distributed kernels function with shape
\begin{equation}\label{eq:UKernel}
C(\xi)=
\begin{cases}
    \frac{|\xi|-\lambda+\delta}{\delta},\quad&|\xi|\ge \lambda-\delta,\\
    0,\quad&|\xi|< \lambda-\delta,
\end{cases}\qquad \lambda>\delta,
\end{equation}
proposed in~\cite{BDFP} in nonlocal unsaturated soil model contexts. \\
Further, we consider tent kernel of the form
\begin{equation}\label{eq:tentKernel}
C(\xi)=\max\{0,\delta-|\xi|\},
\end{equation}
that are commonly considered in typical peridynamic applications, (see for instance~\cite{silling2011}). The kernel functions of interest are depicted in Figure \ref{fig:influenceFunction}.

\begin{figure}
\centering
%
\begin{subfigure}{.32\textwidth}
\begin{tikzpicture}[
  declare function={
    func(\x)= abs(\x);
  }
]
\begin{axis}[
    width = \linewidth,
    axis lines = left,
    xlabel = \(\xi\),
    ylabel = {\(C(\xi)\)},
    ymin = 0,
]

\addplot[
    domain=-10:10,
    samples=200,
    color=black,
] {func(x)};
\end{axis}
\end{tikzpicture}
\end{subfigure}
\begin{subfigure}{.32\textwidth}
\begin{tikzpicture}[
  declare function={
    func(\x)= (abs(\x)-7+1);
  }
]
\begin{axis}[
    width = \linewidth,
    axis lines = left,
    xlabel = \(\xi\),
    ylabel = {\(C(\xi)\)},
    ymin = 0,
]

\addplot[
    domain=-10:10,
    samples=200,
    color=black,
] {func(x)};
\end{axis}
\end{tikzpicture}
\end{subfigure}
\begin{subfigure}{.32\textwidth}
\begin{tikzpicture}[
  declare function={
    func(\x)= (max(0,8-abs(\x));
  }
]
\begin{axis}[
    width = \linewidth,
    axis lines = left,
    xlabel = \(\xi\),
    ylabel = {\(C(\xi)\)},
    ymin = 0,
]

\addplot[
    domain=-10:10,
    samples=200,
    color=black,
] {func(x)};
\end{axis}
\end{tikzpicture}
\end{subfigure}
\caption{Qualitative behaviors of kernel functions defined in 
\eqref{eq:VKernel} with $\lambda=1,\delta=10$, \eqref{eq:UKernel} with $\lambda=7,\delta=1$ and \eqref{eq:tentKernel} with $\delta=8$, respectively.}
\label{fig:influenceFunction}
\end{figure}
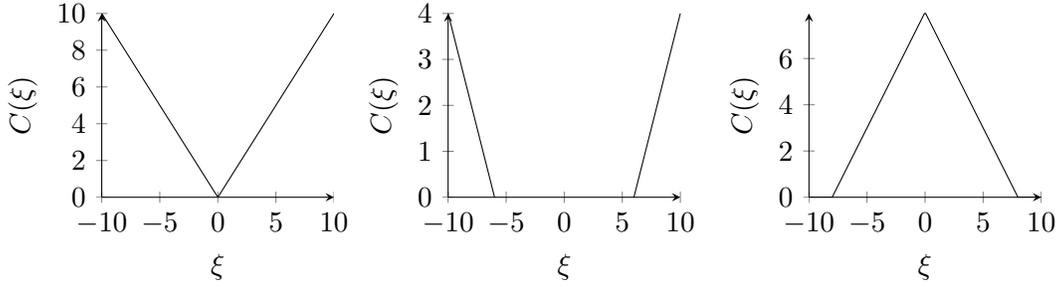

In this paper, we aim to solve the inverse problem described in \eqref{eq:periPDEsupport} for determining the support $[-\delta,\delta]$ of the kernel function $C$, resorting to the learning process provided by a standard Physics Informed Neural Network (PINN).
More specifically, we focus on determining the horizon size $\delta$ of the kernel function within a one-dimensional peridynamic model of a linear microelastic material, testing various kernel types (V-shaped, distributed, and tent) across one- and two-dimensional problems. We provide novel insights into the optimization process, demonstrating a one-sided convergence behavior of the Stochastic Gradient Descent (SGD) optimizer, suggesting that the true horizon value acts as an unstable equilibrium in the PINN gradient flow dynamics. It emphasizes PINN robustness in parameter learning and highlights optimization characteristics unique to the horizon parameter, addressing convergence and stability in PINN optimization for horizon size estimation.
\\
As a consequence, we are not interested in solving the forward problem of determining the solution $u(x,t)$ to \eqref{eq:periPDE}, even though such numerical approximation would be an ancillary product of the proposed PINN. It is worth stressing that the current research differs from \cite{DLP2024} in that here we focus on learning the horizon parameter $\delta$ in a peridynamic context using PINNs, rigorously proving through ad hoc theoretical results the convergence behavior of the SGD method; on the other hand, in \cite{DLP2024} we introduce RBFs to enhance PINN performance for learning the peridynamic kernel function $C(\xi)$, emphasizing physically meaningful solutions, solely focusing on the architectural structure of the serialized PINN proposed to tackle the inverse problem learning the kernel function.

The manuscript is organized as follows. Section~\ref{sec:PINNsoverview} states the problem and describes PINN's architecture we proposed to learn the horizon size of the model. In Section~\ref{sec:convergence} we analyze the relationship between the horizon and the learning process for the PINN realization, proving that the convergence to the horizon limit value, which is a global minimum provided the neural network is wide enough, occurs monotonically if the neural network becomes more insensitive to the parameter change. Section~\ref{sec:results} is devoted to numerical experiments confirming the theoretical results and showing a good capability of the proposed PINN to learn the horizon size for different choice of kernel functions both for 1D and 2D inverse problems. Finally, Section~\ref{sec:conclusion} concludes the paper.

\section{Overview on PINNs}\label{sec:PINNsoverview}

Physics-informed neural networks (PINNs) are a recent advancement that tackle problems governed by partial differential equations (PDEs) (e.g., \cite{Galvanetto2018} for finite element analysis). These architectures integrate physical laws directly into the machine learning framework, offering a promising approach for complex systems. PINNs can be employed for both direct problems (finding solutions with specified initial and boundary conditions) and inverse problems (determining unknown parameters based on observations).

Traditional methods for direct problems, such as finite element analysis (e.g., \cite{Galvanetto2018,ALEBRAHIM2023116034}), finite difference methods with composite quadrature formulas (e.g., \cite{LP2021}), and spectral methods (e.g., \cite{LP,Jafarzadeh,LPeigenv,Jafarzadeh2023}), often require significant computational resources and may loose the sparsity property of the stiffness matrix when applied to nonlocal models. Additionally, these methods might require knowledge of specific material properties (e.g., constitutive parameters, kernel functions) or struggle to enforce certain boundary conditions (e.g., \cite{SUKUMAR2022} proposes PINNs for complex geometries). An alternative approach to traditional methods is given by PINNs, which represent a recent suitable tool to address these issues, yet to be investigated and further deepened, both from a theoretical and a numerical point of view.

Peridynamic theory can also benefit from PINNs. Peridynamic formulations involve integral equations instead of traditional PDEs, and PINNs have been shown effective in solving these integral equations for problems in material characterization \cite{MadenciCMAME2023,DELIA2024,JAFARZADEH2024}. This highlights the versatility of PINNs beyond classical PDE-based problems.

Inverse problems, frequently encountered in real-world applications like medical imaging \cite{ChenEtAl20}, geophysics \cite{bandaiGhezzehei2022}, and material characterization \cite{XU2023,ALEBRAHIM2023116034,JafarzadehArxiv,DLP2024}, are inherently challenging due to potential existence of multiple solutions or no solutions at all. PINNs show promise in overcoming these difficulties, as seen in their application to various inverse problems \cite{Zhou2023,Zunino,MadenciCMAME2023,CaforioEtAl2024}. \\

In this paper we resort to a Feed-Forward fully connected Deep Neural Networks (FF-DNNs or simply NNs), also known as Multi-Layer Perceptrons (MLPs) (see \cite{bengio2003} and references therein). These networks are the results of the concatenation and the arrangement of artificial neurons into layers, and they approximate the solution space through a combination of affine linear maps and nonlinear activation functions $\rho:\R\to\R$ applied across hidden layers, with the independent variable feeding the network's input.

FF-DNNs employ a nested transformation approach where each layer's output serves as the input for the next. \\
Let $L>2$ and let us denote by $[L]\udef\{1,\ldots,L\}$. Mathematically, the realization $\Phi_a(x,\theta)$ of a deep NN with $L$ layers and $N_0$, $N_L$ and $N_l,l\in[L-1]$, representing neurons in the input, output and $l$-th hidden layer respectively, weight matrices $W^{(l)}\in\R^{N_l\times N_{l-1}}$, bias vectors $b\in\R^{N_l}$ and input $x\in\R^{N_0}$, can be expressed as
\begin{equation}\label{eq:NN}
\begin{aligned}
\Phi^{(1)}(x,\theta) &= W^{(1)}x+b^{(1)}, \\
\Phi^{(l+1)}(x,\theta) &= W^{(l+1)}\rho(\Phi^{(l)}(x,\theta))+b^{(l+1)},\quad l\in[L-1] \\
\Phi_a(x,\theta) &= \Phi^{(L)}(x,\theta),
\end{aligned}
\end{equation}
with the activation function $\rho$ being applied componentwise  (see Figure \ref{fig:NNstructure} for a graphical representation of a deep NN). Let us stress that the set of free parameters is
\[
\theta=((W^{(l)},b^{(l)}))_{l=1}^L\in\bigtimes_{l=1}^L\R^{N_l\times N_{l-1}}\times\R^{N_l}\equiv\R^{P(N)},
\]
where $P(N)\udef\sum_{l=1}^{L}N_lN_{l-1}+N_l$ represents the total number of parameters of the NN. Moreover, we define the width of the neural network $\Phi$ as
\[
m\udef\min_{l\in[L]}N_l.
\]
The final output can therefore be obtained by the composition:
\[
\Phi_a(x,\theta)=W^{(L)}\rho(W^{(L-1)}\cdots\rho(W^{(1)}x+b^{(1)})+\ldots+b^{(L-1)})+b^{(L)},\quad x\in\R^{N_0}.
\]
Sometimes, and provided it does not reduce readability, we will hide the dependence of $\Phi_a$ on $\theta$, and will simply write $\Phi_a(x)$. \\
Training PINNs (or, more generally, NNs) amounts to minimizing, with respect to the network's trainable parameters (weights and biases), a loss function that further incorporates the physics of the problem and not only the training data through the Stochastic Gradient Descent (SGD) method.

For a general PDE of the form $\mathcal{P}(u)=0$ (where $\mathcal{P}$ is the differential operator acting on function $u$), the PINN loss function typically takes the form:
\begin{equation}\label{eq:lossPINN}
\mathcal{L}(u,\theta)\udef \mathcal{R}_s(u-u^*,\theta)+\mathcal{R}_d({P}(u)-0^*,\theta),
\end{equation}
where, $u^*$ represents the training data and $0^*$ is the expected value for the differential operation at any training point. The residual functions $\mathcal{R}_s,\mathcal{R}_d$, usually chosen as mean squared error metrics \cite{RAISSIEtAl2019}, depend on the specific problem and functional space; in case of inverse problems, the functions $\mathcal{R}_s,\mathcal{R}_d$ typically depend on the parameter set $\theta$ solely. The first term enforces data fitting, and is referred to as \emph{empirical risk}, while the second term, the differential residual loss, ensures the network adheres to the governing physics. Further terms could be added to \eqref{eq:lossPINN} and enforce other specific properties of the sought solution. We refer to \eqref{eq:PINN_lossFunction} below for the specific form of both empirical risk and differential residual loss, as well as for the selection of $\mathcal{R}_s,\mathcal{R}_d$.

The operator $\mathcal{P}$ is often implemented using automatic differentiation (autodiff) techniques. In the context of peridynamics, a recent work by \cite{HAGHIGHAT2021} proposes a nonlocal alternative to autodiff, utilizing a Peridynamic Differential Operator (PDDO) for evaluating $u$ and its derivatives.

For a recent comprehensive review of PINNs and related theory, we refer to \cite{Cuomo2022}.

\begin{figure}
\centering
\begin{tikzpicture}[x=1.7cm,y=1.4cm]
\message{^^JNeural network, shifted}
\readlist\Nnod{2,4,4,4,1} 
\def\yshift{0.5} 

\message{^^J  Layer}
\foreachitem \N \in \Nnod{ 
\def\lay{\Ncnt} 
\pgfmathsetmacro\prev{int(\Ncnt-1)} 
\message{\lay,}
\foreach \i [evaluate={\c=int(\i==\N); \y=\N/2-\i-\c*\yshift;
             \x=\lay; \n=\nstyle;}] in {1,...,\N}{ 
  \node[node \n] (N\lay-\i) at (\x,\y) {}
  ;
  \ifnum\lay>1 
  \ifnum\lay<4
    \foreach \j in {1,...,\Nnod[\prev]}{ 
      \draw[connect arrow] (N\prev-\j) -- (N\lay-\i);
    }
  \fi
  \fi 
  \ifnum\lay>4
    \foreach \j in {1,...,\Nnod[\prev]}{ 
      \draw[connect arrow] (N\prev-\j) -- (N\lay-\i);
    }
  \fi
  \ifnum\lay=4
    \foreach \j in {1,...,\Nnod[\prev]}{ 
      \draw[dotted] (N\prev-\j) -- (N\lay-\i);
    }
  \fi

}
\ifnum\lay>1 \ifnum\lay<5
\path (N\lay-\N) --++ (0,1+\yshift) node[midway,scale=1.5] {$\vdots$};
\fi\fi
}

\node[left=0.6,align=center,mygreen!60!black] at (N1-1.0) {space\\[-0.2em]input\\[-0.2em]$x$};
\node[left=0.6,align=center,mygreen!60!black] at (N1-2.0) {time\\[-0.2em]input\\[-0.2em]$t$};
\node[below=0.2,align=center,myblue!60!black] at (N3-4.-90)
{
hidden layers\\[-0.2em]$\theta=\{W_1,b_1,\dots,W_l,b_l,\dots,W_L,b_L\}$};
\node[right=0.2,align=center,myred!60!black] at (N\Nnodlen-1.0) {output\\[-0.2em]layer\\[-0.2em]$u_{NN}(x,t;\theta)$};

\end{tikzpicture}
\caption{PINN structure used in this work, with $L$ layers, $N_l$ neurons per layer, $l=0,\ldots,L$.}

\label{fig:NNstructure}
\end{figure}
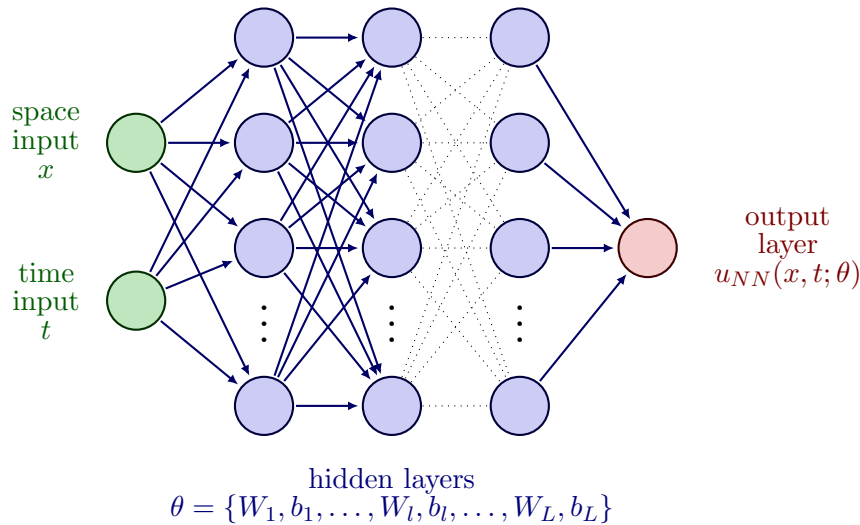

\section{One-sided convergence of the horizon learning process}
\label{sec:convergence}

In this section, we want to analyze how the horizon $\delta$ behaves over the learning process of our PINN realization $\Phi\in\mathcal{F}$, being $\mathcal{F}$ a given class of NN predictors, whose features will be specified later. \\

First, given the training dataset $(x,t,u)\in\R^{N_x}\times\R^{N_t}\times\R^{N_x\times N_t}$, let us rearrange the data, by applying a suitable meshing on $(x,t)$, so that, letting $N\udef N_xN_t$, the neural network realization is the function
\[
\Phi:\R^N\times\R^N\times\R^{P(N)+1}\to\R^N,
\]
where $P(N)$ represents the total number of PINN parameters $\theta=\bmat{\hat\theta \\ \delta}\in\R^{P(N)+1}$, with $\theta_{P(N)+1}\udef\delta\in\R$ and $\hat\theta\in\R^{P(N)}$.
We want to show that the peridynamic model \eqref{eq:periPDE} presents a one-sided convergence for $\delta$, as proved in Theorem \ref{thm:estimate}, and as exemplified by experiments in Section \ref{sec:results}. This will in turn imply that the limit value of the horizon parameter is an unstable equilibrium for the gradient flow process (see, e.g., \cite{Grohs_Kutyniok_2022}) governing $\delta$. \\
Let us then define the loss function \eqref{eq:lossPINN} as
\begin{equation}\label{eq:PINN_lossFunction}
\mathcal{L}(\theta)\udef\frac{1}{2}\left(\sum_{i=1}^{N}|\Phi(x_i,t_i;\theta)-u_i|^2+\sum_{i=1}^{N}|\D(\Phi(x_i,t_i;\theta))|^2\right),
\end{equation}
where, for each input $(x,t)$ in the training dataset, we let the differential residual $\D(\Phi(x,t;\theta))$ be defined as
\begin{equation}\label{eq:PINN_diffRes}
\D(\Phi(x,t;\theta))\udef\parder{^2\Phi}{t^2}(x,t;\theta)-\int_{x-\delta}^{x+\delta}C(x-y)(\Phi(x,t;\theta)-\Phi(y,t;\theta))\,\de y.
\end{equation}
Thus, we want to solve the optimization problem
\begin{equation}\label{eq:minLoss}
\min_{\theta\in\R^{P(N)+1}}\mathcal{L}(\theta),
\end{equation}
with a specific interest for the $(P(N)+1)$-st component of the optimal solution, namely  the parameter $\delta$, representing the peridynamic horizon which, as it will be proven later in this section, is supposed to converge to the true value $\delta^*>0$ we are seeking for. The SGD method applied to the optimization problem \eqref{eq:minLoss} is the iterative process
\begin{equation}\label{eq:sgd}
\theta^{(n+1)}=\theta^{(n)}-\frac{\eta}{2}\left(\nabla_\theta|\Phi(x_i,t_i;\theta^{(n)})-u_i|^2+\nabla_\theta|\D(\Phi(x_i,t_i;\theta^{(n)}))|^2\right),
\end{equation}
where $i$ is uniformly sampled from $\{1,\ldots,N\}$ at each iteration $n\in\N,\,n\geq0$, while $\eta>0$ is the learning rate. \\
In order to perform our analysis, we need some assumptions on the neural network $\Phi$ for which we want an optimal realization relative to \eqref{eq:minLoss}. For sake of simplicity, we will write $\Phi(\theta)$ instead of $\Phi(x,t,\theta)$ if not required by the context. If not otherwise specified, the vector norm is meant to be the Euclidean norm; for matrices, we will make use of the Frobenius norm $\|\cdot\|_\textup{F}$. \\
We first need some definitions.
\begin{defn}\label{def:Lip}
A function $f:\R^p\to\R^q$ is $L_f$-Lipschitz, if there exists $L_f>0$ such that for every $\theta,\sigma\in\R^p$
\[
\|f(\theta)-f(\sigma)\|\leq L_f\|\theta-\sigma\|.
\]
\end{defn}
\begin{defn}\label{def:smooth}
A function $f:\R^p\to\R^q$ is $\beta_f$-smooth if it is differentiable and there exists $\beta_f>0$ such that for every $\theta,\sigma\in\R^p$
\[
\|f(\theta)-f(\sigma)-\nabla f(\theta)(\theta-\sigma)\|\leq \frac{\beta_f}{2}\|\theta-\sigma\|^2.
\]
\end{defn}

If $F$ is smooth enough, then we have an easy sufficient condition to check $\beta$-smoothness.
\begin{lem}\label{lem:smooth}
If a function $f:\R^p\to\R^q$ is twice differentiable, then $f$ is $\|H_f\|_\textup{F}$-smooth, where $H_f$ is the Hessian of $f$.
\end{lem}
\begin{proof}
Letting $\theta,\sigma\in\R^p$, there exists $\xi\in\R^p$ in the segment $\theta,\sigma$ such that
\[
f(\theta)-f(\sigma)=\nabla f(\xi)(\theta-\sigma).
\]
Thus, by Cauchy-Schwarz inequality,
\[
\|f(\theta)-f(\sigma)-\nabla f(\theta)(\theta-\sigma)\|\leq\|\nabla f(\xi)-\nabla f(\theta)\|\|\theta-\sigma\|.
\]
Hence, for some $\overline{\xi}\in\R^p$ in the segment $\theta,\xi$ we have
\[
\nabla f(\xi)-\nabla f(\theta)=\frac12(\xi-\theta)^\top H_{f}(\overline{\xi})(\xi-\theta),
\]
from which
\[
\|\nabla f(\xi)-\nabla f(\theta)\|\leq\frac12\|H_{f}\|_\textup{F}\|\xi-\theta\|^2\leq\frac12\|H_{f}\|_\textup{F}\|\theta-\sigma\|^2.
\]
Therefore
\[
\|f(\theta)-f(\sigma)-\nabla f(\theta)(\theta-\sigma)\|\leq\frac12\|H_{f}\|_\textup{F}\|\theta-\sigma\|^2,
\]
which proves the claim.
\end{proof}
\begin{defn}[Local $\mu$-Polyak-\L{}ojasiewicz condition \cite{LiuEtAl2022}]\label{def:PL}
A nonnegative function $f:\R^p\to\R$ satisfies the $\muPL$ condition on a set $S\subseteq\R^p$ for $\mu>0$ if, for all $\theta\in S$,
\begin{equation}\label{eq:muPL}
\|\nabla f(\theta)\|^2\geq\mu f(\theta).
\end{equation}
\end{defn}

In order to carry our analysis, it is convenient to split the loss function into the \emph{empirical risk}
\begin{equation}\label{eq:empiricalRisk}
\mathcal{R}_s(\theta)\udef\frac{1}{2}\sum_{i=1}^{N}|\Phi(x_i,t_i;\theta)-u_i|^2,
\end{equation}
and the differential residual loss
\begin{equation}\label{eq:differentialLoss}
\mathcal{R}_d(\theta)\udef\frac{1}{2}\sum_{i=1}^{N}|\D(\Phi(x_i,t_i;\theta))|^2,
\end{equation}
so that
\begin{equation}\label{eq:loss}
\mathcal{L}(\theta)=\mathcal{R}_s(\theta)+\mathcal{R}_d(\theta).
\end{equation}

The empirical risk $\mathcal{R}_s$ measures the squared Euclidean norm of the difference between the network prediction $\Phi(x_i,t_i;\theta)$ and synthetic solution $u_i$ over the training mesh. Minimizing this term ensures that the neural network output is close to the given data; moreover, we are enforcing here initial and boundary conditions in the so-called soft way, with the same weight as the one used for the empirical risk over the training mesh. However, this alone does not enforce any physical laws or differential constraints, which is where the differential residual loss $\mathcal{R}_d$ comes into play. It is the squared Euclidean norm of the differential operator applied on the training mesh, where all the derivatives are computed using automatic differentiation. By minimizing this term, the neural network is expected to produce outputs that satisfy the physical law $\mathcal{D}(\Phi(x,t;\theta))=0$.

We are interested in studying the convergence behavior of the horizon $\delta(\tau)$ to $\delta^*$ in \eqref{eq:sgd}. As it will turn out, for a bond-based peridynamic model \eqref{eq:periPDE} convergence occurs under mild assumptions on the differential residual $\D(\Phi(x,t;\theta))$, and it is, further, one-sided. \\

We first focus on the empirical risk $\mathcal{R}_s(\theta)$, whose convergence analysis is standard (see \cite{LiuEtAl2022}).
\begin{prop}\label{prop:lossFunction}
Let us consider the neural network $\Phi(\lambda)$ as given by \eqref{eq:NN}, with a random parameter setting $\theta_0$ such that $\theta_0^{(l)}\sim\mathcal{N}(0,I_{N_l\times N_{l-1}})$ for $l\in[L]$. Let, for $i\in[N]$,
\[
l_i(\theta)\udef\frac12|\Phi(x_i,t_i;\theta)-u_i|^2,
\]
which is twice differentiable, let $H_{l_i}\in\R^{(P(N)+1)\times(P(N)+1)}$ be the Hessian of $l_i$ and let us set
\[
\beta_s\udef\max_{i\in[N]}\|H_{l_i}\|_\textup{F}.
\]
Let the width $m$ of $\Phi(\theta)$ be such that
\[
m=\tilde\Omega\left(\frac{NR_s^{6L+2}}{(\lambda_s-\mu)^2}\right),
\]
where $\lambda_s\udef\lambda_\textup{min}(K(\theta_0))>0$, $K(\theta)\udef\nabla_\theta\Phi(\theta)\nabla_\theta\Phi(\theta)^\top\in\R^{N\times N}$ is the tangent kernel of $\Phi$, $\mu\in(0,\lambda_s)$ is given, and $R_s\udef\frac{2N\sqrt{2\beta_s\mathcal{R}_s(\theta_0)}}{\mu\alpha}$, for some $\alpha\in(0,1)$. \\
Then, with probability $1-\alpha$, letting the step size $\eta\leq\frac{\mu}{N^2\beta_s^2}$ in \eqref{eq:sgd}, SGD relative to $\mathcal{R}_s$ converges to a global solution in the ball $B(\theta_0;R_s)$, with an exponential convergence rate:
\[
\mathbb{E}[\mathcal{R}_s(\theta^{(n)})]\leq\left(1-\frac{\mu\eta}{N}\right)^n\mathcal{R}_s(\theta_0).
\]
\end{prop}
\begin{proof}
From Lemma \ref{lem:smooth}, $l_i$ is $\beta_s$-smooth for each $i\in[N]$ since they are twice differentiable. Moreover, because of the hypothesis on the width $m$, $\mathcal{R}_s$ satisfies the $\muPL$ condition in $B(\theta_0;R_s)$ (see \cite[Theorem 4]{LiuEtAl2022}). Therefore, from \cite[Theorem 7]{LiuEtAl2022}, the claim follows.
\end{proof}

Next, we prove that also the differential residual $\mathcal{R}_d$ converges to zero, with high probability, over the training phase.
\begin{prop}\label{prop:diffRes}
Let, for $i\in[N]$,
\[
d_i(\theta)\udef\frac12|\D(\Phi(x_i,t_i,\theta))|^2,
\]
which is twice differentiable, let $H_{d_i}\in\R^{(P(N)+1)\times(P(N)+1)}$ be the Hessian of $d_i$ and let us set
\[
\beta_d\udef\max_{i\in[N]}\|H_{d_i}\|_\textup{F}.
\]
Moreover, let $R_d\udef\frac{2N\sqrt{2\beta_d\mathcal{R}_d(\theta_0)}}{\mu\alpha}$, for some $\alpha\in(0,1)$, where $\mu\in(0,\lambda_d)$ is given, being $\lambda_d\udef\lambda_\textup{min}\left(\D(\nabla_\theta\Phi(\theta_0))\D(\nabla_\theta\Phi(\theta_0))^\top\right)$. For all $\theta\in B(\theta_0;R_d)$, let us assume the following:
\begin{align}
&\D\left(\parder{\Phi}{\hat\theta}\right)\in\R^{N\times N}\textrm{ is full rank}, \label{eq:ass:fullRank} \\
&\D\left(\parder{\Phi}{\delta}\right)^\top\Phi\leq\frac12\|\Phi\|^2. \label{eq:ass:peri}
\end{align}
Then, with probability $1-\alpha$, letting the step size $\eta\leq\frac{\mu}{N^2\beta_d^2}$ in \eqref{eq:sgd}, SGD relative to $\mathcal{R}_d$ converges to a global solution in the ball $B(\theta_0;R_d)$, with an exponential convergence rate:
\[
\mathbb{E}[\mathcal{R}_d(\theta^{(n)})]\leq\left(1-\frac{\mu\eta}{N}\right)^n\mathcal{R}_d(\theta_0).
\]
\end{prop}
\begin{proof}
Let $\theta\in B(\theta_0;R_d)$ be given. From Lemma \ref{lem:smooth}, the functions $d_i$ are $\beta_d$-smooth for each $i\in[N]$ since they are twice differentiable. \\
Let us now observe that the matrix $\D(\nabla_\theta\Phi)$ can be partitioned as
\[
\D(\nabla_\theta\Phi)=\bmat{\D\left(\parder{\Phi}{\hat\theta}\right) & \D\left(\parder{\Phi}{\delta}\right)},
\]
so that
\[
\D(\nabla_\theta\Phi)\D(\nabla_\theta\Phi)^\top=\D\left(\parder{\Phi}{\hat\theta}\right)\D\left(\parder{\Phi}{\hat\theta}\right)^\top+\D\left(\parder{\Phi}{\delta}\right)\D\left(\parder{\Phi}{\delta}\right)^\top.
\]
Since $\D\left(\parder{\Phi}{\hat\theta}\right)$ is full rank, $\D(\nabla_\theta\Phi)\D(\nabla_\theta\Phi)^\top$ is positive definite. Therefore
\[
\lambda_\textup{min}(\D(\nabla_\theta\Phi)\D(\nabla_\theta\Phi)^\top)>0.
\]
Let us now compute $\parder{\mathcal{R}_d}{\delta}(\theta)$. Letting
\[
\varphi_\Phi(y)\udef C(x-y)(\Phi(x,t)-\Phi(y,t)),\quad y\in(x-\delta,x+\delta),
\]
for any $\delta\neq\delta^*$, $\delta>0$, we have
\begin{align*}
\parder{}{\delta}\left(\int_{x-\delta}^{x+\delta}\varphi_\Phi(y)\de y\right) &= \parder{}{\delta}\left(\int_{x-\delta}^{x+\delta}C(x-y)\Phi(x)\de y-(C*\Phi(\cdot,t))(x)\right) \\
&= \parder{}{\delta}\left(\Phi(x)\int_{x-\delta}^{x+\delta}C(x-y)\de y-(C*\Phi(\cdot,t))(x)\right) \\
&= \parder{}{\delta}\left(\delta\Phi(x)-(C*\Phi)(x)\right) \\
&= \Phi(x,t)+\delta\parder{\Phi}{\delta}(x)-(C*\Phi(\cdot,t))(x) \\
&= \Phi(x)+\int_{x-\delta}^{x+\delta}C(x-y)\left(\parder{\Phi}{\delta}(x,t)-\parder{\Phi}{\delta}(y,t)\right)\de y,
\end{align*}
where the convolution product $(C*\Phi(\cdot,t))(x)$ is supported over $[x-\delta,x+\delta]$. Thus, from \eqref{eq:PINN_diffRes} it follows that
\begin{align*}
\parder{\mathcal{R}_d}{\delta}(\theta) &= \left\langle\parder{\D}{\delta}(\Phi),\D(\Phi;\delta)\right\rangle \\
&= \left\langle\parder{}{\delta}\parder{^2\Phi}{t^2}-\parder{}{\delta}\left(\int_{x-\delta}^{x+\delta}\parder{\varphi_\Phi}{\delta}(y)\de y\right),\D(\Phi)\right\rangle \\
&= \left\langle\parder{}{\delta}\parder{^2\Phi}{t^2}-\Phi-\int_{x-\delta}^{x+\delta}C(x-y)\left(\parder{\Phi}{\delta}(x,t)-\parder{\Phi}{\delta}(y,t)\right)\de y,\D(\Phi)\right\rangle \\
&= \left\langle\D\left(\parder{\Phi}{\delta}\right)-\Phi,\D(\Phi)\right\rangle.
\end{align*}
Therefore, letting $\Phi_i(\theta)\udef\Phi(x_i,t_i;\theta)$ for $i\in[N]$, we have that
\begin{align*}
\frac12\|\nabla_\theta\mathcal{R}_d(\theta)\|^2 &= \frac{1}{2}\left(\sum_{j=1}^{N(P)}\left(\sum_{i=1}^N\D(\Phi_i)\D\left(\parder{\Phi_i}{\theta_j}\right)\right)^2+
\left(\sum_{i=1}^N\D(\Phi_i)\left(\D\left(\parder{\Phi_i}{\delta}\right)-\Phi_i\right)\right)^2\right) \\
&= \frac{1}{2}\left(\D(\Phi)^\top\D(\nabla_\theta\Phi)D(\nabla_\theta\Phi)^\top\D(\Phi)+\D(\Phi)^\top\mathcal{A}\D(\Phi)\right),
\end{align*}
where
\[
\mathcal{A}\udef\hat{\mathcal{A}}+\hat{\mathcal{A}}^\top,\quad\hat{\mathcal{A}}\udef\Phi\left(\frac12\Phi-\D\left(\parder{\Phi}{\delta}\right)\right)^\top.
\]
Now, $\hat{\mathcal{A}}$ is a rank 1 matrix, whose unique nonzero eigenvalue is equal to $\left(\frac12\Phi-\D\left(\parder{\Phi}{\delta}\right)\right)^\top\Phi$, that is nonnegative because of \eqref{eq:ass:peri}. Therefore $\hat{\mathcal{A}}$ is nonnegative definite, and so is $\mathcal{A}$, which is further symmetric. This implies that $\D(\Phi)^\top\mathcal{A}\D(\Phi)\geq0$, and hence
\begin{align*}
\frac12\|\nabla_\theta\mathcal{R}_d(\theta)\|^2 &\geq \frac{1}{2}\D(\Phi)^\top\D(\nabla_\theta\Phi)D(\nabla_\theta\Phi)^\top\D(\Phi) \\
&\geq \lambda_\textup{min}(\D(\nabla_\theta\Phi)D(\nabla_\theta\Phi)^\top)\frac{1}{2}\|\D(\Phi)\|^2 \\
&\geq \mu\mathcal{R}_d(\theta),
\end{align*}
saying that $\mathcal{R}_d$ satisfies the $\muPL$ condition in $B(\theta_0;R_d)$. Therefore, again from \cite[Theorem 7]{LiuEtAl2022}, the claim follows.
\end{proof}

Let us now observe that, under the hypothesis of Proposition \ref{prop:lossFunction} and Proposition \ref{prop:diffRes}, it is reasonable to expect that the realization $\Phi(\theta)$ would be more and more insensitive to the parameter $\delta$ as $\theta$ approaches the global minimum in some suitably small neighborhood of $\theta_0$. Therefore, we will assume that
\begin{equation}\label{eq:dPhidDeltaTo0}
\lim_{n\to\infty}\parder{\Phi(\theta^{(n)})}{\delta}=0,
\end{equation}
where $\theta^{(n)}$ is evolving according to the SGD method \eqref{eq:sgd}.
\begin{lem}\label{lem:dPhidDeltaTo0}
Let $\Phi$ be given as in \eqref{eq:NN}. Then
\[
\lim_{n\to\infty}\D\left(\parder{\Phi(\theta^{(n)})}{\delta}\right)=0.
\]
\end{lem}
\begin{proof}
The claim follows from the smoothness of $\Phi$, since the activation function is smooth, and the nature of the differential operator $\D$.
\end{proof}

Now, let $\{\theta_s^{(n)}\}_{n\in\N},\{\theta_d^{(n)}\}_{n\in\N}$ be the two sequences arising from Proposition \ref{prop:lossFunction} and Proposition \ref{prop:diffRes}, relative to $\mathcal{R}_s,\mathcal{R}_d$ and convergent to $\theta_s^*,\theta_d^*$ respectively, within the ball $B(\theta_0;R)$, where $R\udef\min\{R_s,R_d\}$. Let us further assume that such global minima are unique in $B(\theta_0;R)$.

It is straightforward that, if $\theta_s^*=\theta_d^*=\theta^*$, then such a common value $\theta^*$ is a minimum point for $\mathcal{L}(\theta)$. \\ However, this is typically not the case and, in order to broach the optimization problem \eqref{eq:minLoss}, we propose to consider the following multi-objective problem:
\begin{equation}\label{eq:mop}
\min_{\theta\in\R^{P(N)+1}}\mathcal{L}_m(\theta)=\bmat{\mathcal{R}_s(\theta) \\ \mathcal{R}_d(\theta)}.
\end{equation}
This way, as a consequence of \eqref{eq:loss}, problem \eqref{eq:minLoss} can be seen as a linear scalarization version (we refer to \cite{EmmerichDeutz2018} for a comprehensive review on the topic) of \eqref{eq:mop} with uniform weights. \\
Before carrying out our analysis, let us recall  some definitions.
\begin{defn}\label{def:paretoDom}
Let $x,y\in\R^p$. We say that $x$ Pareto-dominates $y$ and we write $x\prec y$ if and only if $x_i\leq y_i$ for all $i\in[p]$ and $x_i<y_i$ for at least one $i\in[p]$.
\end{defn}
\begin{defn}\label{def:paretoEff}
Let $f:\R^p\to\R^q$ and let us consider the multi-objective problem $\min_{x\in\R^p}f(x)$. We say that a solution $x\in\R^p$ is Pareto optimal if and only if there does not exist $y\in\R^p$ such that $f(y)\prec f(x)$.
\end{defn}
It holds the following.
\begin{prop}
The global minimum solutions $\theta_s^*,\theta_d^*$ are Pareto optimal for $\mathcal{L}_m(\theta)$ in $B(\theta_0;R)$.
\end{prop}
\begin{proof}
Let $\overline\theta\in\R^{P(N)+1}$ and let us assume that $\mathcal{L}_m(\overline{\theta})\prec\mathcal{L}_m(\theta_s^*)$. Therefore $\mathcal{R}_s(\overline{\theta})\leq\mathcal{R}_s(\theta_s^*)$, $\mathcal{R}_d(\overline{\theta})\leq\mathcal{R}_d(\theta_s^*)$ and at least one of them holds strictly. Since $\theta_s^*$ is the unique global minimum for $\mathcal{R}_s$ in $B(\theta_0;R)$, then $\overline{\theta}=\theta_s^*$ and  $\mathcal{R}_d(\overline{\theta})<\mathcal{R}_d(\theta_s^*)$, which is a contradiction. Therefore $\theta_s^*$ is Pareto optimal and so is, with analogous computations, $\theta_d^*$.
\end{proof}
We want to prove now that the SGD \eqref{eq:sgd} relative to the linear scalarization problem \eqref{eq:minLoss} indeed converges to a global minimum in a suitable neighborhood of $\theta_0$. Since, if it exists, this minimum point will be Pareto optimal (see \cite[Proposition 8]{EmmerichDeutz2018}), then we have to expect that $\nabla\mathcal{R}_s,\nabla\mathcal{R}_d$ would compete around the minimum. In fact, we are going to prove that, if such a competition between gradients is bounded from below, then $\mathcal{L}$ satisfies a $\muPL$, and hence the convergence will follow.
\begin{thm}\label{thm:L_globalMinConvergence}
Let all the assumptions of Proposition \ref{prop:lossFunction} and Proposition \ref{prop:diffRes} hold. Moreover, let $\beta\udef\min\{\beta_s,\beta_d\}$ and $R_0\udef\min\{R_s,R_d\}$. Let us further assume that
\begin{equation}\label{eq:gradCompetition}
\langle\nabla\mathcal{R}_s(\theta),\nabla\mathcal{R}_d(\theta)\rangle>-\frac{\mu}{2}\mathcal{L}(\theta)
\end{equation}
for all $\theta\in B(\theta_0;R_0)$ and for some $\mu\in(0,\min\{\lambda_s,\lambda_d\})$. \\
Then there exist $\overline\mu>0$ and $R\in(0,R_0)$ such that, for some $\alpha\in(0,1)$, letting the step size $\eta\leq\frac{\overline\mu}{N^2\beta^2}$ in \eqref{eq:sgd}, with probability $1-\alpha$ the SGD relative to $\mathcal{L}$ converges to a global solution in the ball $B(\theta_0;R)$, with an exponential convergence rate:
\[
\mathbb{E}[\mathcal{L}(\theta^{(n)})]\leq\left(1-\frac{\overline\mu\eta}{N}\right)^n\mathcal{L}(\theta_0).
\]
\end{thm}
\begin{proof}
From \eqref{eq:gradCompetition}, defining
\[
\overline{\mu}\udef\mu+2\min_{\theta\in B(\theta_0;R_0)}\frac{\langle\nabla\mathcal{R}_s(\theta),\nabla\mathcal{R}_d(\theta)\rangle}{\mathcal{L}(\theta)},
\]
it follows that $\overline{\mu}>0$. Now, let us set
\[
R\udef\min\left\{R_0,\frac{2N\sqrt{2\beta\mathcal{L}(\theta_0)}}{\overline\mu\alpha}\right\}.
\]
Because of \eqref{eq:loss}, for $\theta\in B(\theta_0;R)$ we have that
\begin{align*}
\|\nabla\mathcal{L}(\theta)\|^2 &=\|\nabla\mathcal{R}_s(\theta)+\nabla\mathcal{R}_d(\theta)\|^2 \\
&= \|\nabla\mathcal{R}_s(\theta)\|^2+\|\nabla\mathcal{R}_d(\theta)\|^2+2\langle\nabla\mathcal{R}_s(\theta),\nabla\mathcal{R}_d(\theta)\rangle \\
&\geq \mu\mathcal{L}(\theta)+2\langle\nabla\mathcal{R}_s(\theta),\nabla\mathcal{R}_d(\theta)\rangle \\
&\geq \overline{\mu}\mathcal{L}(\theta),
\end{align*}
implying that $\mathcal{L}(\theta)$ satisfies the $\muPL$ condition in $B(\theta_0;R)$. Resorting to \cite[Theorem 7]{LiuEtAl2022} proves the claim.
\end{proof}
\begin{cor}\label{cor:paretoOptL}
The global solution to \eqref{eq:minLoss} is Pareto optimal for $\mathcal{L}(\theta)$ on $B(\theta_0;R)$.
\end{cor}
\begin{proof}
Since \eqref{eq:loss} is a linear scalarization of $\mathcal{L}_m(\theta)$, then from \cite[Proposition 8]{EmmerichDeutz2018} we deduce that the global solution of Theorem \ref{thm:L_globalMinConvergence} is Pareto optimal for $\mathcal{L}(\theta)$ on $B(\theta_0;R)$.
\end{proof}
We can now prove the following result about one-sided convergence of $\{\delta^{(n)}\}_{n\in\N}$ to $\delta^*$.
\begin{thm}\label{thm:estimate}
Let all the assumptions of Theorem \ref{thm:L_globalMinConvergence} hold, let $\alpha\in(0,1)$, and let $\varepsilon>0$ be given. Then, there exists $\nu>0$ such that, for all $n>\nu$:
\begin{itemize}
\item
if $\mathbb{E}[\Phiin\D(\Phiin)]>2\varepsilon^\frac32$, then with probability $1-\alpha$: $\mathbb{E}[\delta^{(n+1)}]>\mathbb{E}[\delta^{(n)}]$;
\item
if $\mathbb{E}[\Phiin\D(\Phiin)]<-2\varepsilon^\frac32$, then with probability $1-\alpha$: $\mathbb{E}[\delta^{(n+1)}]<\mathbb{E}[\delta^{(n)}]$.
\end{itemize}
\end{thm}
\begin{proof}
Looking at the $(n+1)$st component of \eqref{eq:sgd} and performing analogous computations as in the proof of Proposition \ref{prop:diffRes}, we obtain that
\begin{align*}
\delta^{(n+1)} &= \delta^{(n)}-\frac{\eta}{2}\left(\parder{}{\delta}|\Phiin-u_i|^2+\parder{}{\delta}|\D(\Phiin)|^2\right) \\
&= \delta^{(n)}-\frac{\eta}{2}\left((\Phi_i^{(n)}-u_i)\parder{\Phiin}{\delta}+\D(\Phiin)\left(\D\left(\parder{\Phiin}{\delta}\right)-\Phiin\right)\right).
\end{align*}
From Theorem \ref{thm:L_globalMinConvergence} there exists $\nu>0$ such that, for all $n>\nu$, and using Jensen's inequality,
\[
\mathbb{E}[|\Phiin-u_i|]^2\leq\mathbb{E}[|\Phiin-u_i|^2]\leq\mathbb{E}[\mathcal{R}_s(\theta^{(n)})]\leq\mathbb{E}[\mathcal{L}(\theta^{(n)})]\leq\varepsilon,
\]
and
\[
\mathbb{E}[|\D(\Phiin)|]^2\leq\mathbb{E}[\D(\Phiin)^2]\leq\mathbb{E}[\mathcal{R}_d(\theta^{(n)})]\leq\mathbb{E}[\mathcal{L}(\theta^{(n)})]\leq\varepsilon.
\]
Therefore
\begin{align*}
\mathbb{E}[|\Phiin-u_i|] &\leq \varepsilon^\frac12, \\
\mathbb{E}[|\D(\Phiin)|] &\leq \varepsilon^\frac12.
\end{align*}
Also, from \eqref{eq:dPhidDeltaTo0} and Lemma \ref{lem:dPhidDeltaTo0}, we have
\begin{align*}
\left|\parder{\Phiin}{\delta}\right| &\leq \varepsilon, \\
\left|\D\left(\parder{\Phiin}{\delta}\right)\right| &\leq \varepsilon.
\end{align*}
Hence, it follows that
\begin{align*}
\mathbb{E}[\delta^{(n+1)}-\delta^{(n)}] &= \eta\Bigg(\mathbb{E}[\Phiin\D(\Phiin)] \\
&\quad -(\Phiin-u_i)\parder{\Phiin}{\delta}-\D(\Phiin)\D\left(\parder{\Phiin}{\delta}\right)\Bigg),
\end{align*}
and thus
\[
\eta\left(\mathbb{E}[\Phiin\D(\Phiin)]-2\varepsilon^\frac32\right)\leq\mathbb{E}[\delta^{(n+1)}-\delta^{(n)}]\leq\eta\left(\mathbb{E}[\Phiin\D(\Phiin)]+2\varepsilon^\frac32\right),
\]
from which the claim follows.
\end{proof}

\begin{rem}
Theorem \ref{thm:estimate} says that the convergence of $\{\delta^{(n)}\}$ to the global minimum, whose existence is guaranteed by Proposition \ref{prop:lossFunction} and Proposition \ref{prop:diffRes}, under the condition in \eqref{eq:dPhidDeltaTo0}, must be monotonic. Such a behavior has been observed and reported in Section \ref{sec:results}. However, it seems that, in the 1D case, the convergence is monotonically decreasing, while it is monotonically increasing in the 2D case. We are not able to say anything about that this would always be the case or even why, and it will be further investigated.
\end{rem}

\begin{rem}\label{rem:norm}
If we replace the means in the loss function \eqref{eq:PINN_lossFunction} with norms and define
\begin{equation}\label{eq:PINN_lossFunction_norm}
\mathcal{L}_2(\Phi,\delta)\udef\sqrt{\sum_{i=1}^{N_x}\sum_{j=1}^{N_t}|\Phi(x_i^*,t_j^*)-\theta_{ij}^*|^2}+\sqrt{\sum_{i=1}^{N_x}\sum_{j=1}^{N_t}|\D(\Phi(x_i^*,t_j^*);\delta^*)|^2},
\end{equation}
then the analysis above should be slightly modified, as we would have denominators when taking derivatives with respect to the parameters, that would go to zero as $\delta$ approaches $\delta^*$. The analysis, in this case, seems more elusive, as reported in Section \ref{sec:results}. In fact, we can notice that the minimization process suffers from stagnation at some unreliably high level either for the data loss, or for the residual loss, or for both. We surmise that, in this case, the convergence in the minimization process of $\mathcal{L}_2$ could tend towards a Pareto optimal solution, which is not a global minimum, still one-sided as in the case of relative to $\mathcal{L}$ as loss function.
\end{rem}

\begin{rem}
If $\delta^*$ is not known a priori and one has no hint on how to select a suitable initial guess, there is no guarantee that $\delta$ would converge towards the true value $\delta^*$. In fact, Proposition \ref{prop:lossFunction}, Proposition \ref{prop:diffRes} and Theorem \ref{thm:estimate} provide local results, and the attractivity region depends on quantities that are often hard to compute, so that convergence could get stuck at some local minimum. This is an interesting and deep aspect worth to be further investigated.
\end{rem}

\section{Numerical Experiments}\label{sec:results}

In this section we present several experiments to show how PINNs behave in the context of inverse problems in bond-based peridynamic models relative to the learning process of the horizon parameter. It is interesting to notice that, with standard tuning of loss functions, learning rate and PINN architecture, such problems are relatively well-conditioned in some suitable convergence region. More specifically, we are going to see that such regions are usually one-sided, possibly suggesting that the true sought values are unstable equilibrium points for the PINN model gradient flow. \\
The PINN architecture used in the next examples has a representation with 8 hidden layers, each made up by 20 neurons; the activation function is $\tanh$, and a \texttt{glorot\_normal} kernel initializer acts on each layer (including the output layer); moreover, we implemented ADAM optimizer for our experiments.

The machine used for the experiments is an Intel Core i7-8850H CPU at 2.60GHz and 64 GB of RAM; the code has been written in Python 3.10, using TensorFlow 2.15.0 within the Keras 3.0.1. \\

In the next examples we show how convergence is attained, when solving \eqref{eq:minLoss}, for different kernel shapes, only if the training process is started from a superestimate of $\delta^*$. Moreover, we report the convergence issues reported in Remark \ref{rem:norm} relative to the $\mathcal{L}_2$.

\begin{exm}\label{ex:data2}
In Section~\ref{sec:convergence} we proved that the horizon learning process is one-sided convergent, in the sense that, for one-dimensional problems, the method can attain the horizon size value only if we start the process with an initial value greater than the expected one. This example aims to provide a numerical confirmation of the theoretical result presented in the previous section. \\
Let us consider a kernel function of type \eqref{eq:VKernel}, whose expression is given by
\[
C(\xi)=
\begin{cases}
    \frac35|\xi|,\quad&|\xi|\geq\delta^*,\\
    0,\quad&|\xi|<\delta^*,
\end{cases}
\]
with $\delta^*=10$. Letting
\[
c(\xi)\udef\frac35|\xi|,
\]
we notice that we can globally rewrite $C(\xi)$, for every $\xi\in\R$, as
\[
C(\xi)=c_\textup{min}(\xi)+c(\delta)\sgn(c_\textup{min}(\xi)),\quad c_\textup{min}(\xi)\udef\min\{c(\xi)-c(\delta),0\}.
\]

\begin{figure}[htb]
\centering
\includegraphics[width=0.32\textwidth]{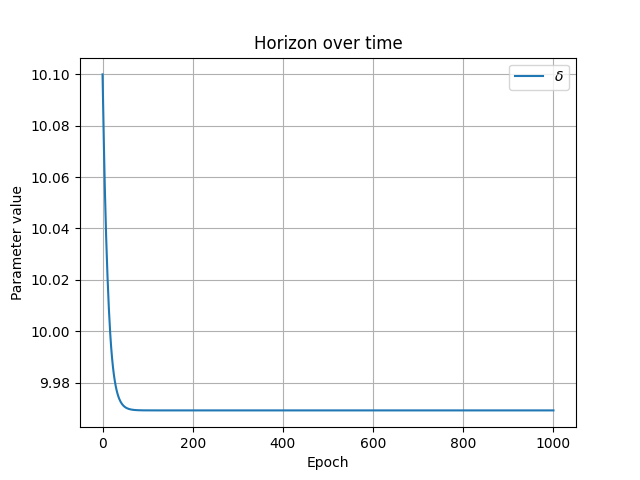}
\includegraphics[width=0.32\textwidth]{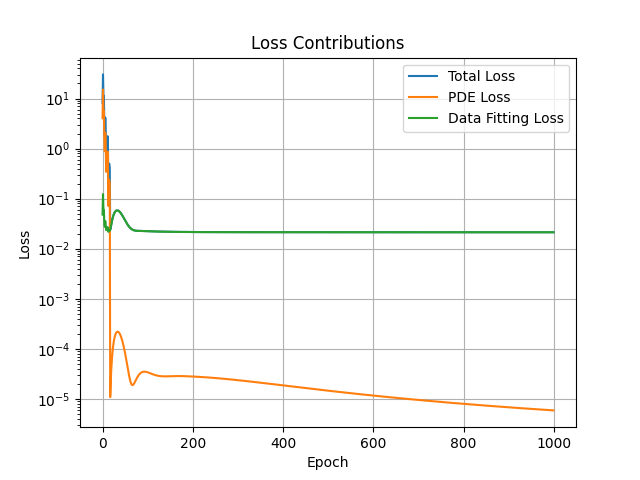}
\includegraphics[width=0.32\textwidth]{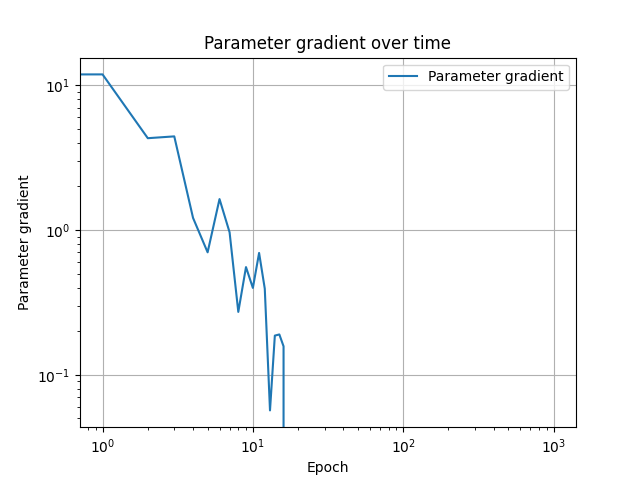}
\caption{Parameter learning, loss and gradient evolution for Example \ref{ex:data2} starting at $\delta=10.1$; the last graph is in logarithmic scale. The true value for the parameter is $\delta^*=10$. The loss function is the mean squared empirical risk $\mathcal{L}$ in \eqref{eq:PINN_lossFunction} with constant learning rate.}
\label{fig:data2inval1_1_meanSquared}
\end{figure}
\begin{figure}[htb]
\centering
\includegraphics[width=0.32\textwidth]{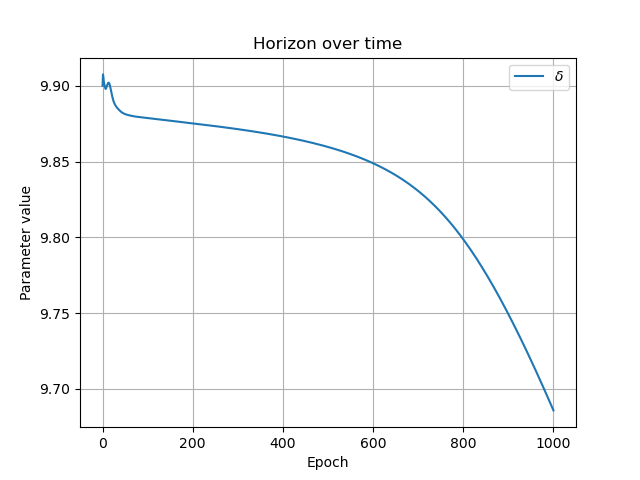}
\includegraphics[width=0.32\textwidth]{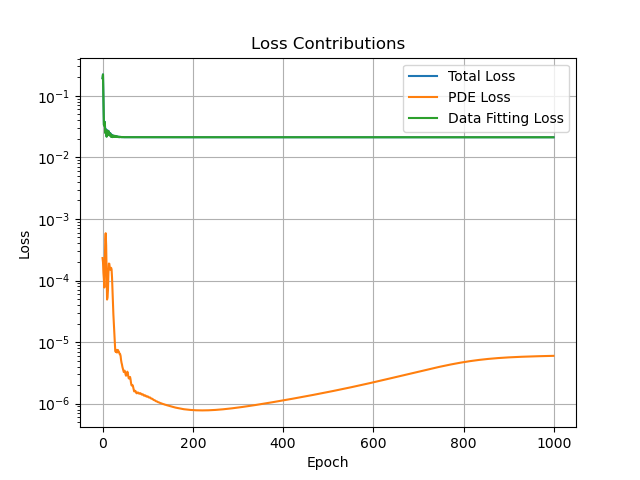}
\includegraphics[width=0.32\textwidth]{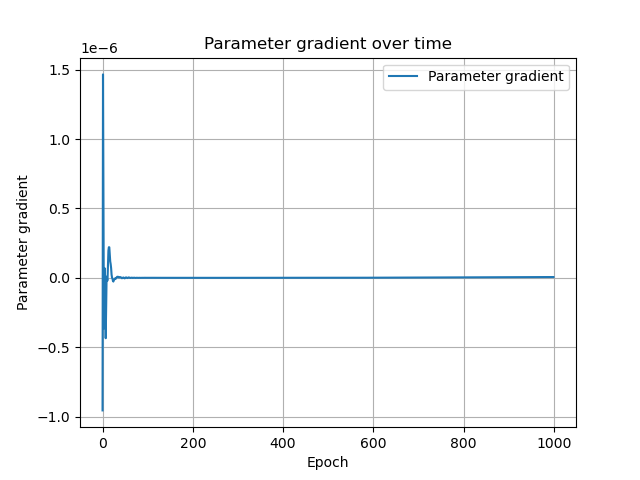}
\caption{Parameter learning, loss and gradient evolution for Example \ref{ex:data2} starting at $\delta=9.9$. The true value for the parameter is $\delta^*=10$. The loss function is the mean squared empirical risk $\mathcal{L}$ in \eqref{eq:PINN_lossFunction} with constant learning rate.}
\label{fig:data2inval9_9_meanSquared}
\end{figure}

We perform two simulations with the same setting, but changing the initial guess. In the first case, we start the process by an initial condition greater than $\delta^\ast$ and we observe that the process converges to $\delta^\ast$. While, for initial values belonging to a left neighborhood of $\delta^\ast$ the convergence of the process is not guaranteed.
Figure \ref{fig:data2inval1_1_meanSquared} is obtained with an initial guess $\delta=10.1$; as it can be seen from the leftmost graph, the gradient stays positive and goes to zero, providing convergence. \\
We also performed an analogous simulation with a starting value $\delta=9.9<\delta^*$. In this case, as shown in Figure \ref{fig:data2inval9_9_meanSquared}, there is no evidence of convergence to some stable value in 1000 epochs. In the rightmost graph, the gradient evolution stays positive after a transient of sign changing.  \\

Moreover, we report experimental results about convergence issues when the loss function is $\mathcal{L}_2$ as in \eqref{eq:PINN_lossFunction_norm}. In Figure \ref{fig:data2inval11_norm} we chose an initial superestimate $\delta=11$ for $\delta^*$; as it can be seen from the rightmost graph, the gradient stays positive and goes to zero, providing convergence for the residual loss, while the empirical risk seems to be not minimized at $\delta^*$, suggesting the process may have reached a Pareto optimal solution that is not a global minimum. \\
We also performed an analogous simulation with a starting value $\delta=9.8<\delta^*$. In this case, as shown in Figure \ref{fig:data2inval9.8}, there is no evidence of convergence to some stable value in 1000 epochs; again, the residual loss seems to stagnate, suggesting that some other equilibrium could exist, different from $\delta^*$ and isolated. \\

For all the simulations relative to this example, the learning rate has been kept constant to $1e-2$ over the 1000 epochs of the training process.

\begin{figure}[htb]
\centering
\includegraphics[width=0.32\textwidth]{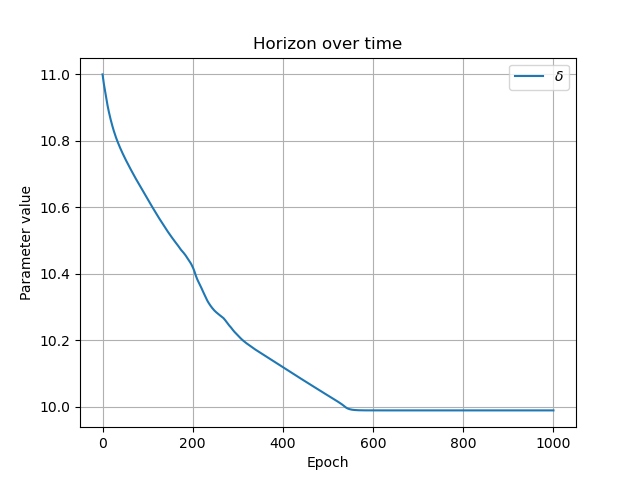}
\includegraphics[width=0.32\textwidth]{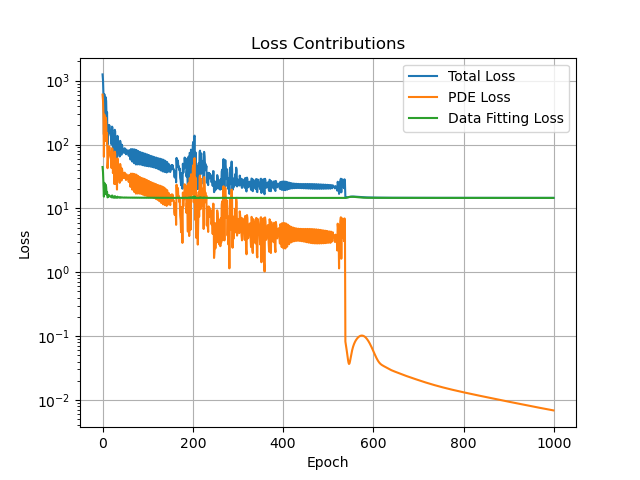}
\includegraphics[width=0.32\textwidth]{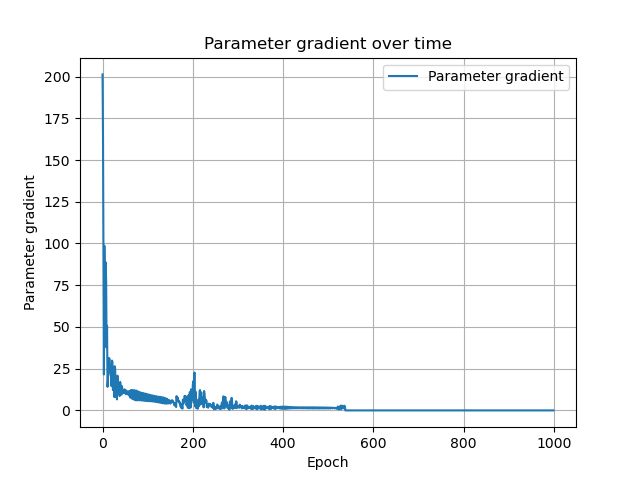}
\caption{Parameter learning evolution, loss and gradient for Example \ref{ex:data2} starting at $\delta=11$. The true value for the parameter is $\delta^*=10$. The loss function is the euclidean norm empirical risk $\mathcal{L}_2$ in \eqref{eq:PINN_lossFunction_norm} with constant learning rate.}
\label{fig:data2inval11_norm}
\end{figure}

\begin{figure}[htb]
\centering
\includegraphics[width=0.32\textwidth]{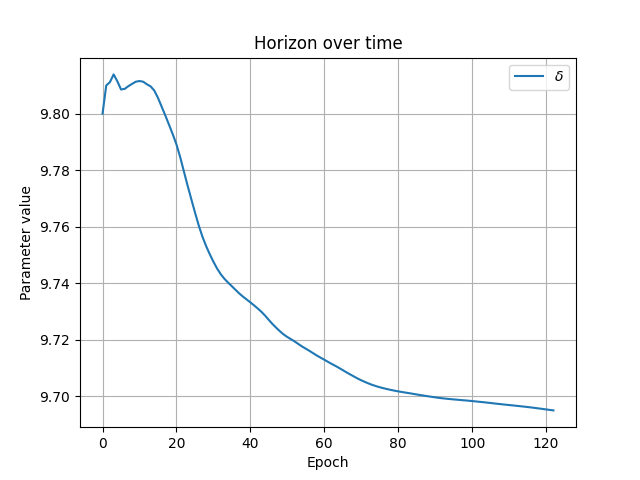}
\includegraphics[width=0.32\textwidth]{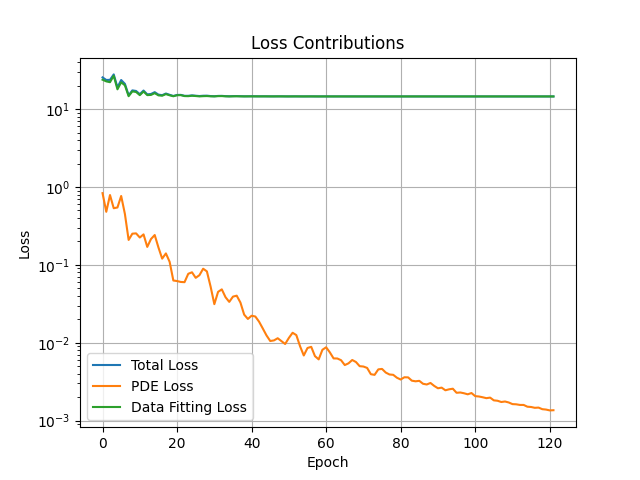}
\includegraphics[width=0.32\textwidth]{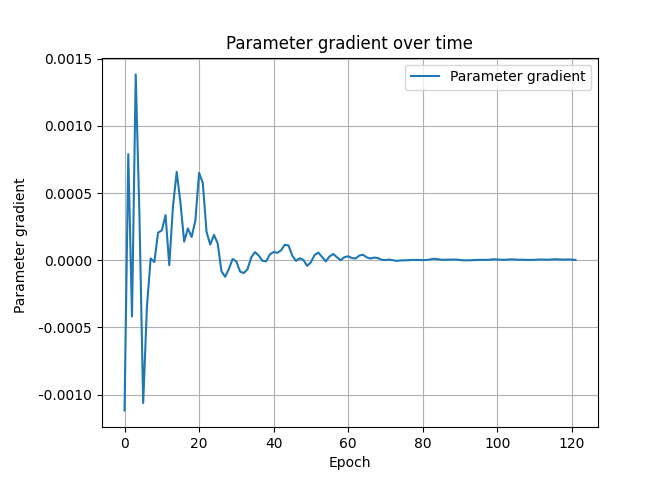}
\caption{Parameter learning evolution for Example \ref{ex:data2} starting at $\delta=9.8<\delta^*=10$. The loss function is the euclidean norm empirical risk $\mathcal{L}_2$ in \eqref{eq:PINN_lossFunction_norm} with constant learning rate.}
\label{fig:data2inval9.8}
\end{figure}

\end{exm}
\begin{exm}\label{ex:data3}
In this example, a kernel function of type \eqref{eq:UKernel} is considered, with expression
\[
C(\xi)=
\begin{cases}
	\frac{|\xi|-10+\delta^*}{\delta^*},\quad&|\xi|\geq10-\delta^*,\\
    0,\quad&|\xi|<10-\delta^*,
\end{cases}
\]
with $\delta^*=1$. Letting
\begin{align*}
c(\xi) &\udef \left|\frac{\xi}{\delta}\right|+\frac{\delta-10}{\delta}, \\
c_0(\xi) &\udef \max\{c(\xi),0\},
\end{align*}
analogously as in Example \ref{ex:data2}, we can rewrite $C(\xi)$, for every $\xi\in\R$, as
\[
C(\xi)=c_\textup{min}(\xi)+c_0(\delta)\sgn(c_\textup{min}(\xi)),\quad c_\textup{min}(\xi)\udef\min\{c_0(\xi)-c_0(\delta),0\}.
\]

In Figure \ref{fig:data3_inVal1_5_meanSquared} we show convergence of the horizon towards a good approximation of the true value starting from $\delta=1.5$. We selected a constant learning rate set to $1e-2$ over the 1000 epochs of the training process.

\begin{figure}[htb]
\centering
\includegraphics[width=0.32\textwidth]{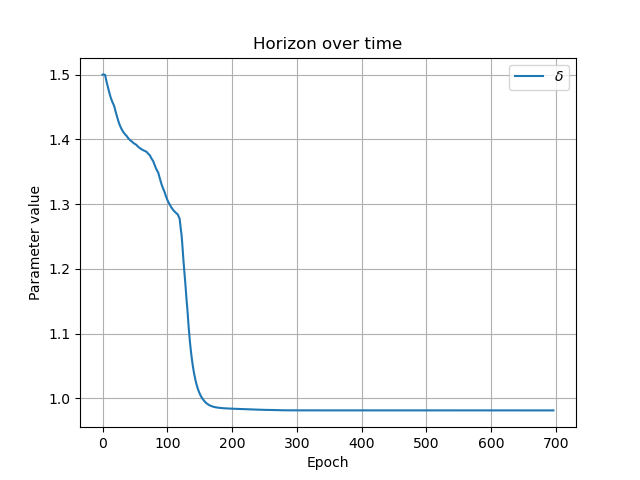}
\includegraphics[width=0.32\textwidth]{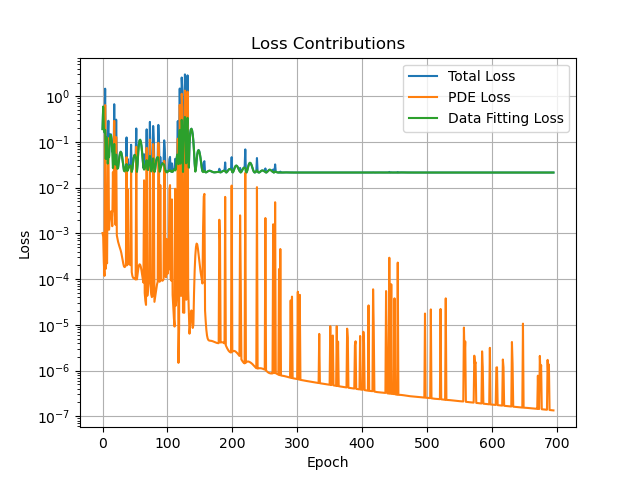}
\includegraphics[width=0.32\textwidth]{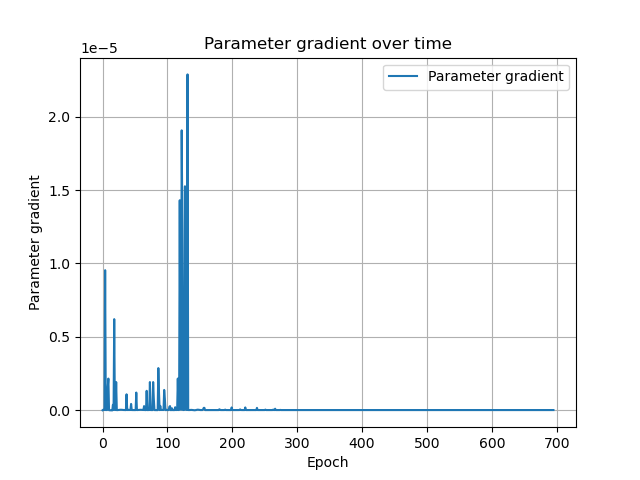}
\caption{Parameter learning, loss and gradient evolution for Example \ref{ex:data3} starting at $\delta=1.5$. The true value for the parameter is $\delta^*=1$. The loss function is $\mathcal{L}$ and the the learning rate is constant, set at $1e-2$.}
\label{fig:data3_inVal1_5_meanSquared}
\end{figure}

For this case, we also experimented on different learning rate. More specifically, when using a cyclical \texttt{PolynomialDecay} scheduler of degree 3, with an initial value of $1e-2$ decaying to a final value of $1e-4$ every 100 epochs, over a total number of 1000 epochs, we obtain the results shown in Figure \ref{fig:data3_inVal1_5_meanSquared_polyDecay}. \\

\begin{figure}[htb]
\centering
\includegraphics[width=0.32\textwidth]{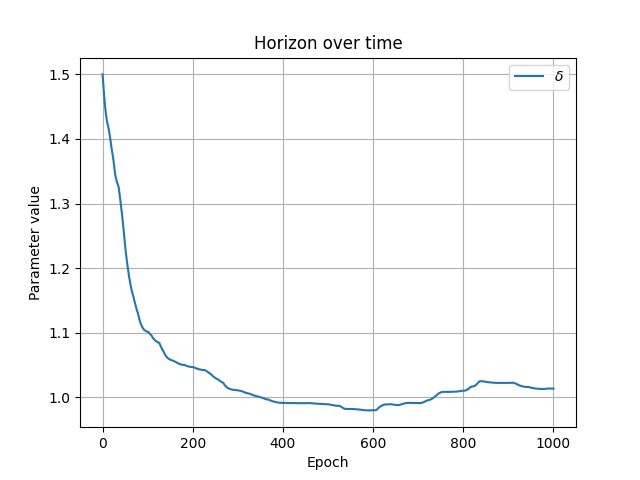}
\includegraphics[width=0.32\textwidth]{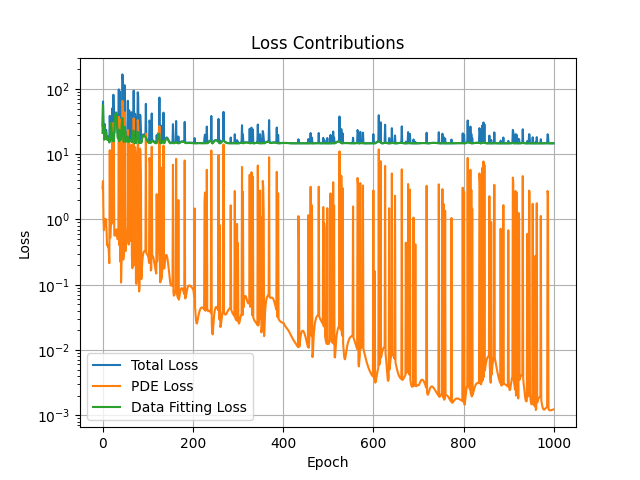}
\includegraphics[width=0.32\textwidth]{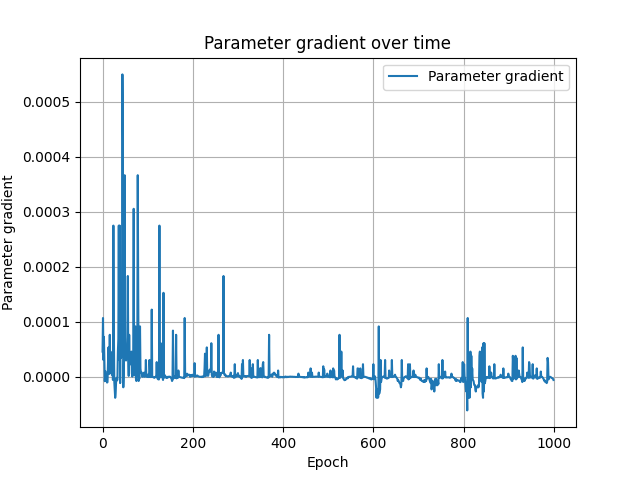}
\caption{Parameter learning, loss and gradient evolution for Example \ref{ex:data3} starting at $\delta=1.5$. The true value for the parameter is $\delta^*=1$. The loss function is $\mathcal{L}$ and the the learning rate follows a polynomial decay.}
\label{fig:data3_inVal1_5_meanSquared_polyDecay}
\end{figure}

For both previous cases, when starting at a subestimate $\delta=0.9<\delta^*$, we obtain a monotone divergence from the true value, as depicted in Figures \ref{fig:data3_inVal0_9_meanSquared} and \ref{fig:data3_inVal0_9_meanSquared_polyDecay}, where a constant learning rate and a polynomial decay has been chosen, respectively, with the same settings used for Figures \ref{fig:data3_inVal1_5_meanSquared} and \ref{fig:data3_inVal1_5_meanSquared_polyDecay}.

\begin{figure}[htb]
\centering
\includegraphics[width=0.32\textwidth]{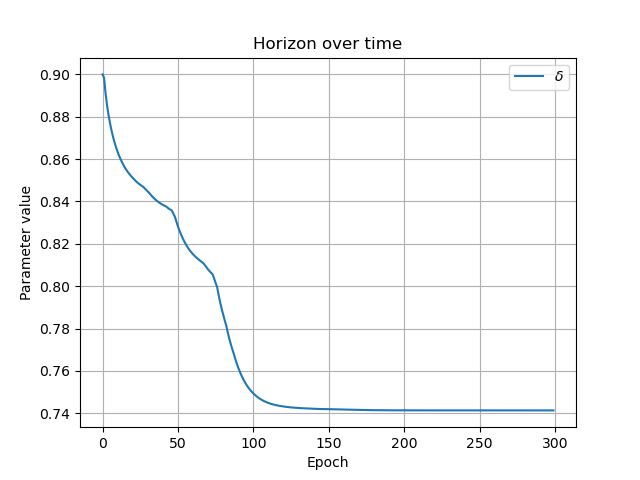}
\includegraphics[width=0.32\textwidth]{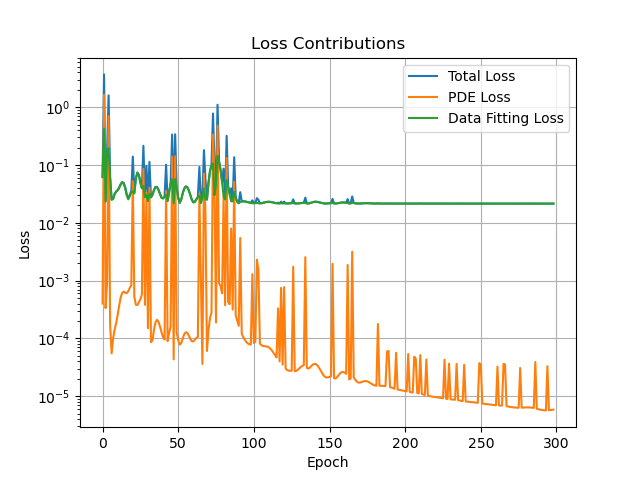}
\includegraphics[width=0.32\textwidth]{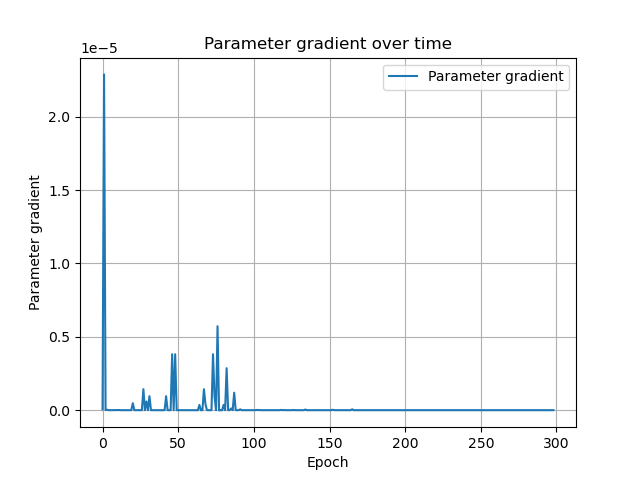}
\caption{Parameter learning, loss and gradient evolution for Example \ref{ex:data3} starting at $\delta=0.9$. The true value for the parameter is $\delta^*=1$. The loss function is $\mathcal{L}$ as in \eqref{eq:PINN_lossFunction} and the the learning rate is constant, set at $1e-2$.}
\label{fig:data3_inVal0_9_meanSquared}
\end{figure}

\begin{figure}[htb]
\centering
\includegraphics[width=0.32\textwidth]{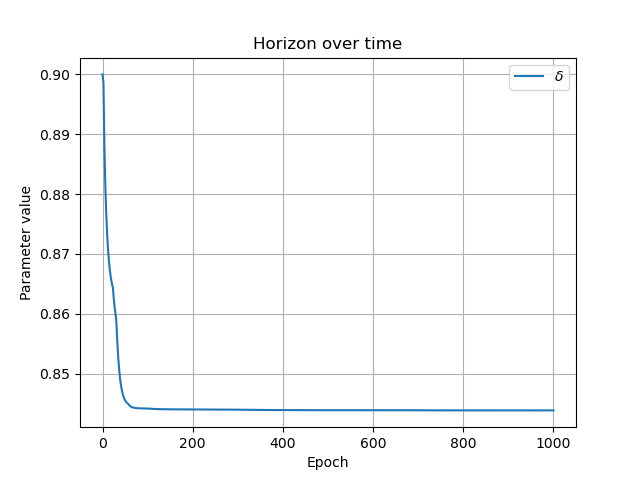}
\includegraphics[width=0.32\textwidth]{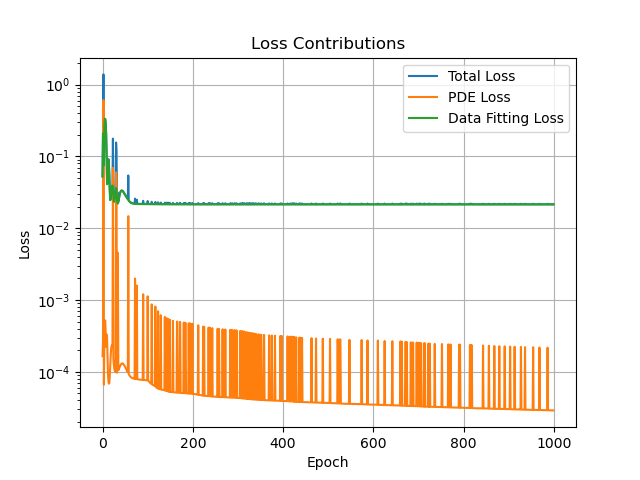}
\includegraphics[width=0.32\textwidth]{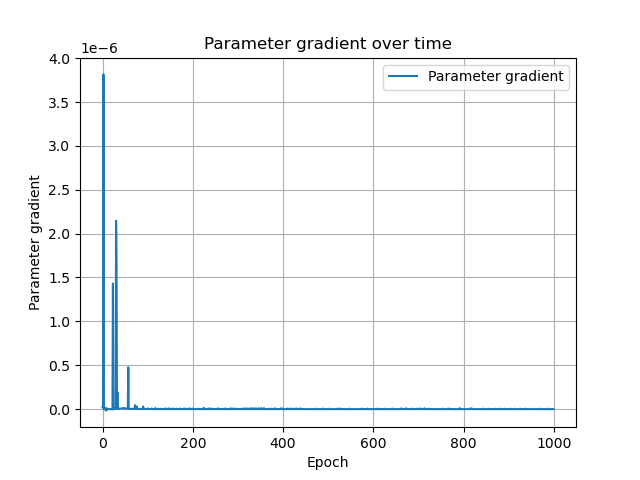}
\caption{Parameter learning, loss and gradient evolution for Example \ref{ex:data3} starting at $\delta=1.5$. The true value for the parameter is $\delta^*=1$. The loss function is $\mathcal{L}$ and the the learning rate follows a polynomial decay.}
\label{fig:data3_inVal0_9_meanSquared_polyDecay}
\end{figure}

Figure \ref{fig:data3_inVal1_5_norm} is obtained with an initial guess $\delta=1.5$ when minimizing the loss function $\mathcal{L}_2$ as in \eqref{eq:PINN_lossFunction_norm}. Here, a cyclical \texttt{PolynomialDecay} scheduler of degree 5 has been used for the learning rate, with an initial value of $1e-2$ decaying to a final value of $1e-4$ every 100 epochs, over a total number of 1000 epochs. As in previous example, the residual loss seems to be not minimized at $\delta^*$, suggesting the process may have reached a Pareto optimal that is not a global minimum.

\begin{figure}[htb]
\centering
\includegraphics[width=0.32\textwidth]{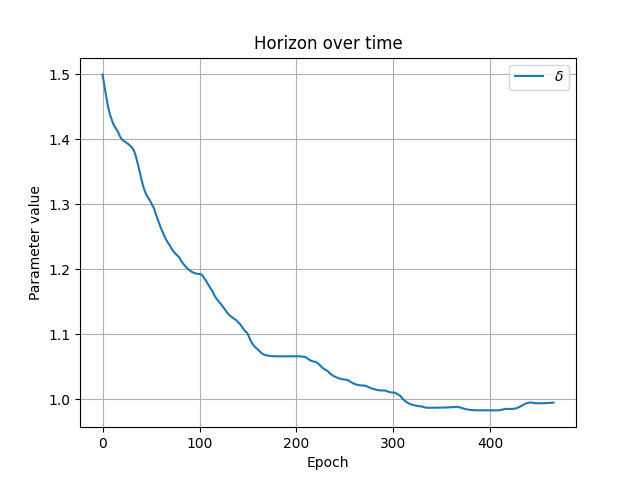}
\includegraphics[width=0.32\textwidth]{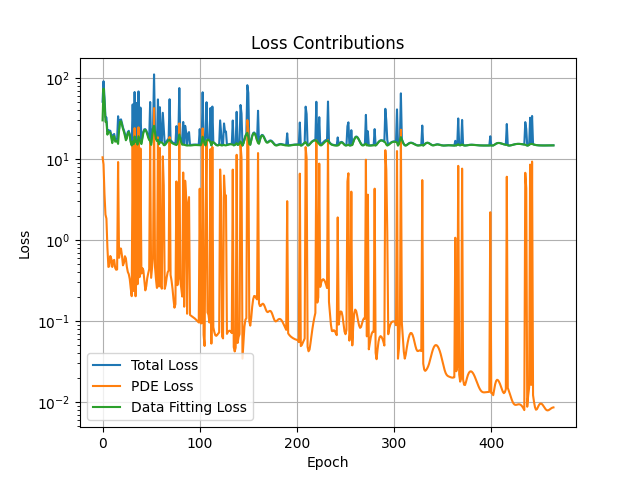}
\includegraphics[width=0.32\textwidth]{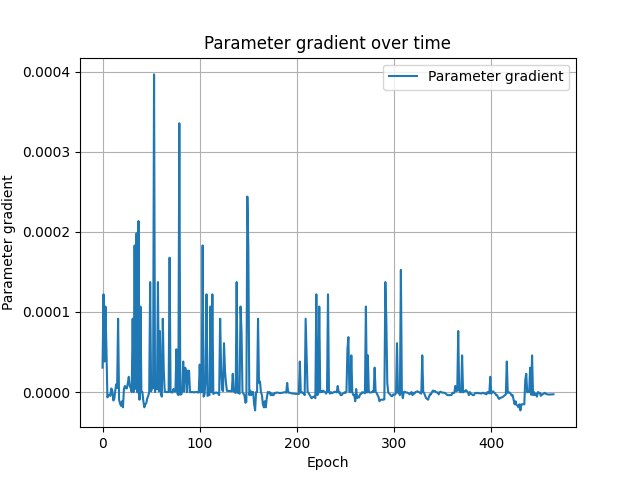}
\caption{Parameter learning evolution for Example \ref{ex:data3} with an initial guess $\delta=1.5$; the true value for the parameter is $\delta^*=1$. The loss function is $\mathcal{L}_2$ as in \eqref{eq:PINN_lossFunction_norm}, minimized using a polynomial decaying learning rate.}
\label{fig:data3_inVal1_5_norm}
\end{figure}

Again, starting at $\delta=0.9$, below $\delta^*=1$ ends up in divergence with an unreasonably high magnitude residual loss, as depicted in Figure \ref{fig:data3_inVal0_9_norm}, when the same polynomial decaying learning rate has been used is previous simulations.

\begin{figure}[htb]
\centering
\includegraphics[width=0.32\textwidth]{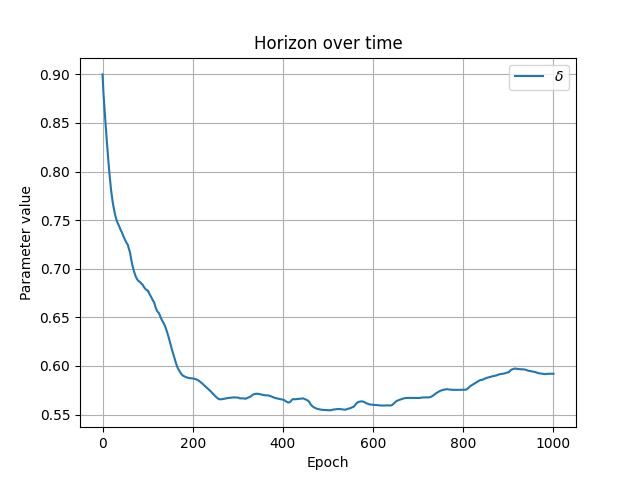}
\includegraphics[width=0.32\textwidth]{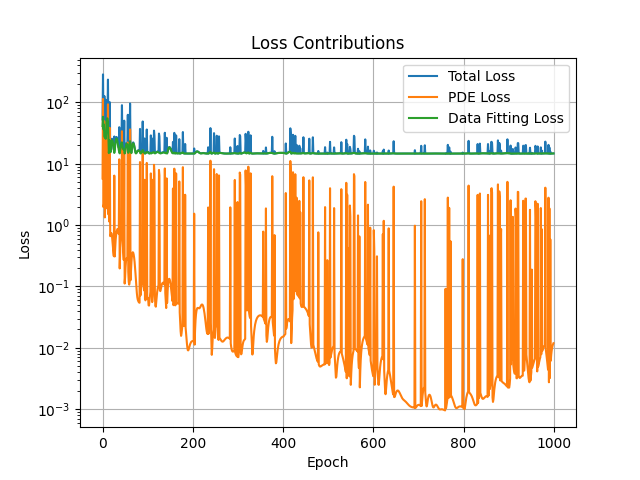}
\includegraphics[width=0.32\textwidth]{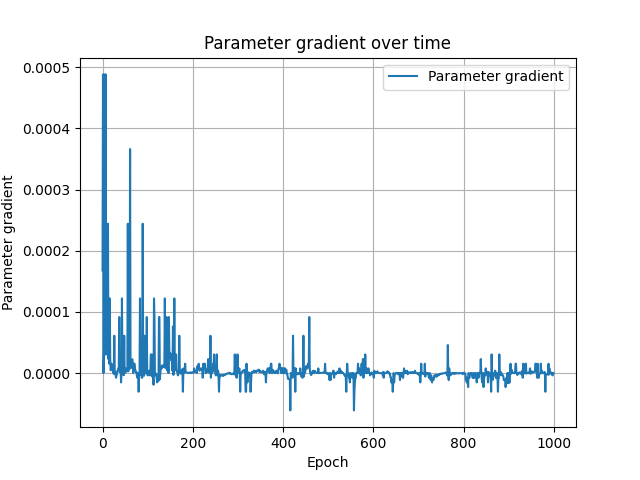}
\caption{Parameter learning evolution for Example \ref{ex:data3} with an initial guess $\delta=0.9$; the true value for the parameter is $\delta^*=1$. The loss function is $\mathcal{L}_2$ as in \eqref{eq:PINN_lossFunction_norm} and the learning rate follows a polynomial decay.}
\label{fig:data3_inVal0_9_norm}
\end{figure}

\end{exm}

\begin{exm}\label{ex:data8}
In this example, a kernel function of type \eqref{eq:tentKernel} is considered, with expression
\[
C(\xi)=\max\{0,\delta^*-|\xi|\}
\]
where $\delta^*=1$. \\
Figure \ref{fig:data8_inVal1_1_meanSquared} shows the convergence of the horizon to $\delta^*=1$ when starting at a superestimate $\delta=1.1$ and minimizing $\mathcal{L}$ as in \eqref{eq:PINN_lossFunction}; when minimizing $\mathcal{L}_2$ as in \eqref{eq:PINN_lossFunction_norm}, we get behaviors shown in Figure \ref{fig:data8_inVal1_1_norm}, where we can witness again what reported in Remark \ref{rem:norm}. For these results, a learning rate following a \texttt{CosineDecay} scheduler has been selected, setting the initial value to $1e-4$, decay steps equal to the number of epochs and no warm-up step.

\begin{figure}[htb]
\centering
\includegraphics[width=0.32\textwidth]{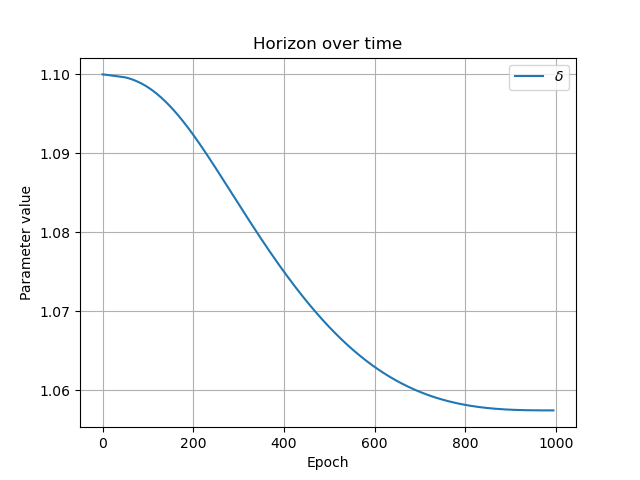}
\includegraphics[width=0.32\textwidth]{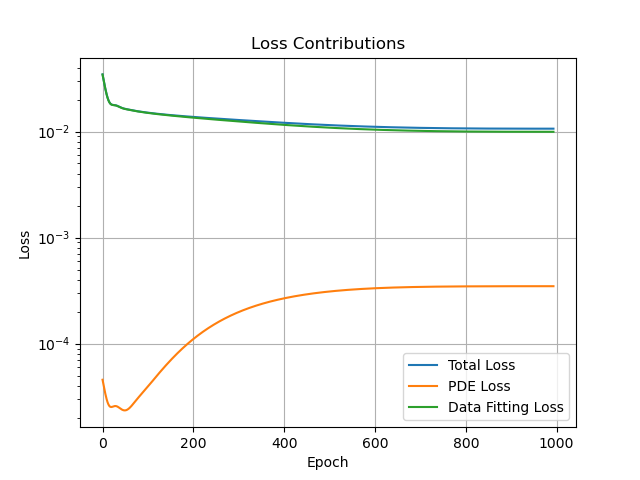}
\includegraphics[width=0.32\textwidth]{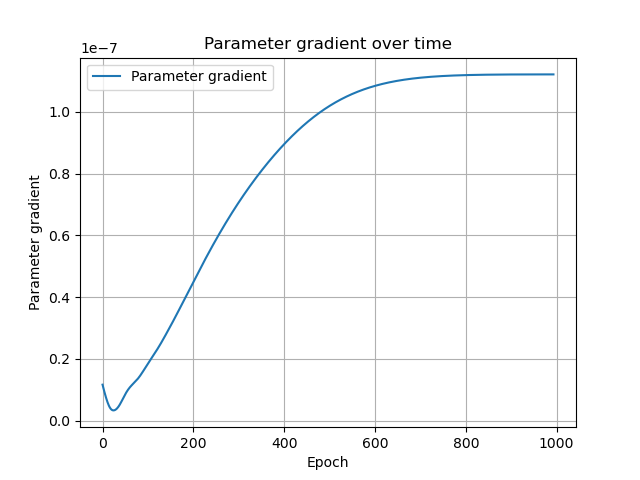}
\caption{Parameter learning evolution for Example \ref{ex:data8} with an initial guess $\delta=1.1$; the true value for the parameter is $\delta^*=1$. The loss function is $\mathcal{L}$ as in \eqref{eq:PINN_lossFunction} and the learning rate follows a cosine decay.}
\label{fig:data8_inVal1_1_meanSquared}
\end{figure}

\begin{figure}[htb]
\centering
\includegraphics[width=0.32\textwidth]{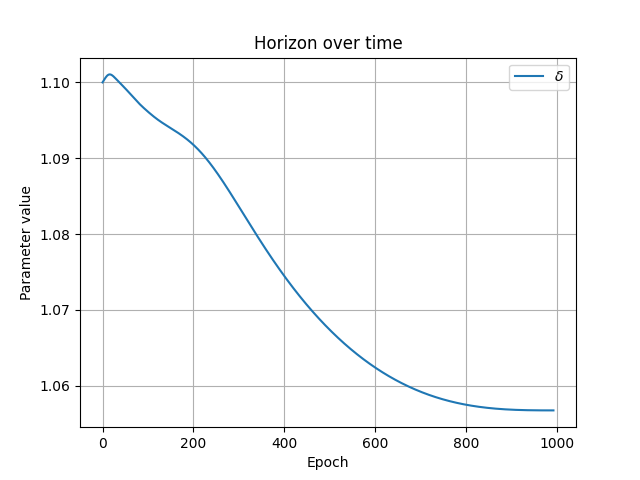}
\includegraphics[width=0.32\textwidth]{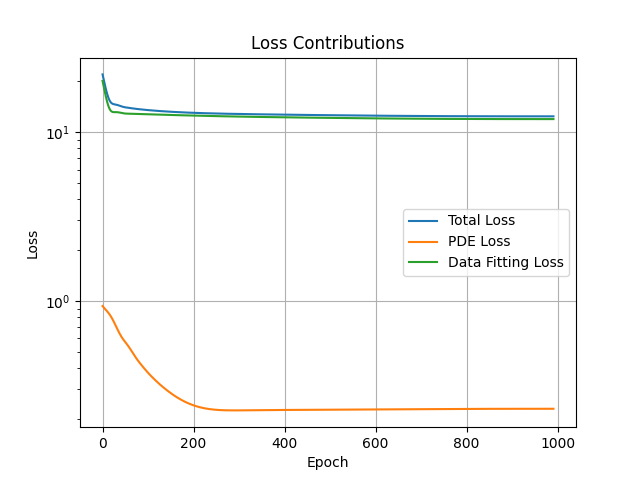}
\includegraphics[width=0.32\textwidth]{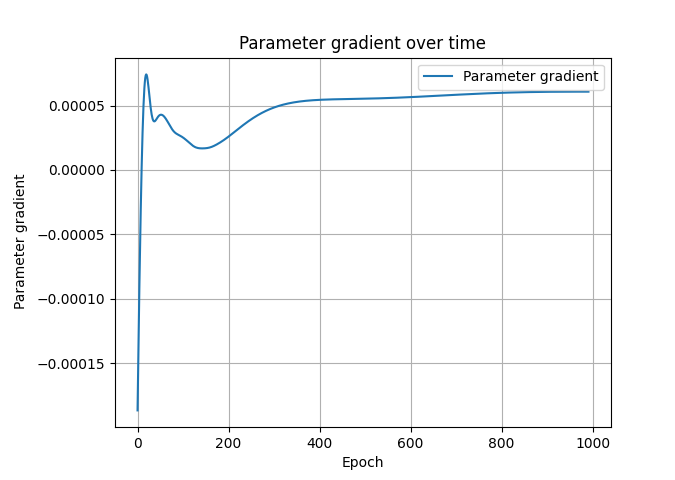}
\caption{Parameter learning evolution for Example \ref{ex:data8} with an initial guess $\delta=1.1$; the true value for the parameter is $\delta^*=1$. The loss function is $\mathcal{L}_2$ as in \eqref{eq:PINN_lossFunction_norm} and the learning rate follows a cosine decay.}
\label{fig:data8_inVal1_1_norm}
\end{figure}

Within the same setting as above, starting from a subestimate $\delta=0.9$ provides divergence, as depicted in Figure \ref{fig:data8_inVal0_9_meanSquared} for the minimization of $\mathcal{L}$, and in Figure \ref{fig:data8_inVal0_9_norm} for the minimization of $\mathcal{L}_2$.

\begin{figure}[htb]
\centering
\includegraphics[width=0.32\textwidth]{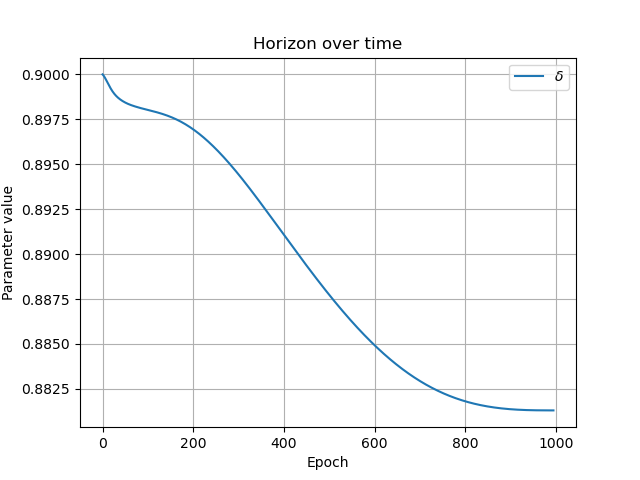}
\includegraphics[width=0.32\textwidth]{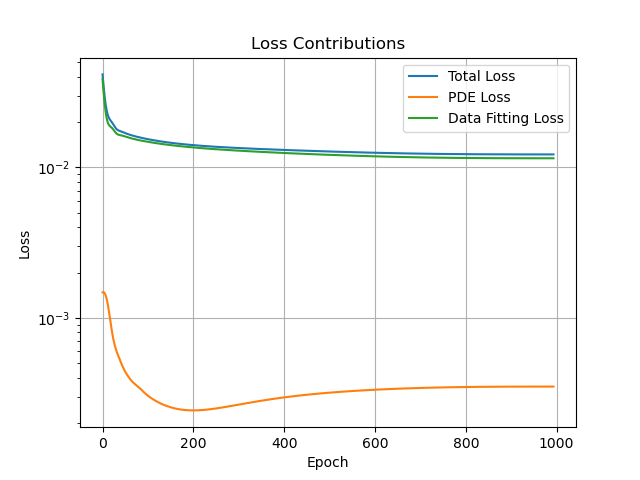}
\includegraphics[width=0.32\textwidth]{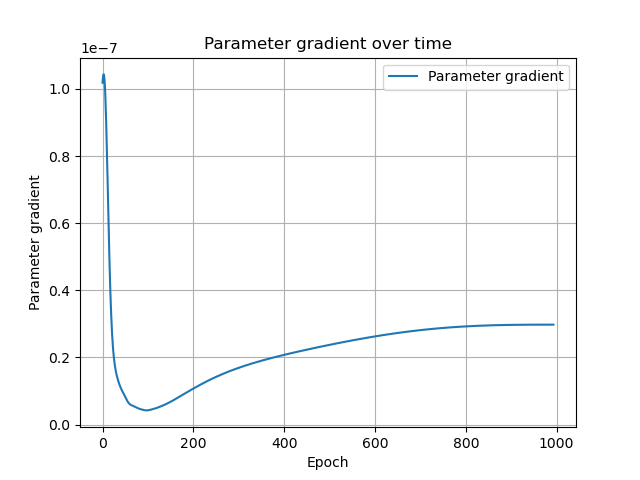}
\caption{Parameter learning evolution for Example \ref{ex:data8} with an initial guess $\delta=0.9$; the true value for the parameter is $\delta^*=1$. The loss function is $\mathcal{L}$ as in \eqref{eq:PINN_lossFunction} and the learning rate follows a cosine decay.}
\label{fig:data8_inVal0_9_meanSquared}
\end{figure}

\begin{figure}[htb]
\centering
\includegraphics[width=0.32\textwidth]{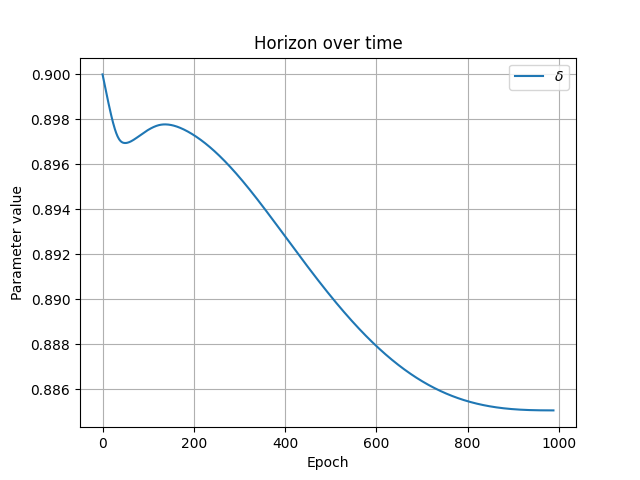}
\includegraphics[width=0.32\textwidth]{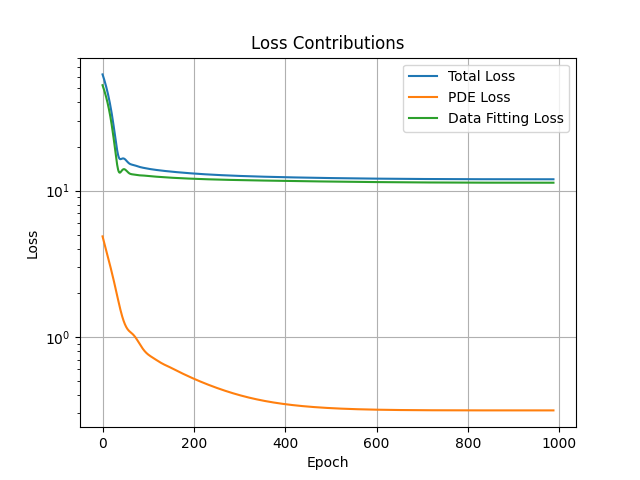}
\includegraphics[width=0.32\textwidth]{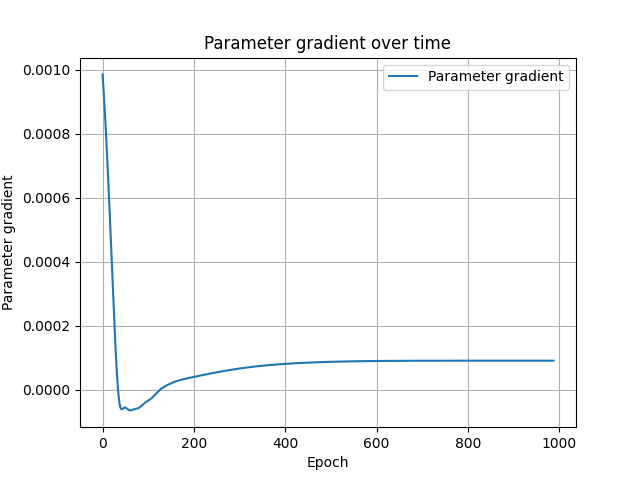}
\caption{Parameter learning evolution for Example \ref{ex:data8} with an initial guess $\delta=0.9$; the true value for the parameter is $\delta^*=1$. The loss function is $\mathcal{L}_2$ as in \eqref{eq:PINN_lossFunction_norm} and the learning rate follows a cosine decay.}
\label{fig:data8_inVal0_9_norm}
\end{figure}

\end{exm}

From previous example, we witness that convergence is attained starting from a superestimate of the true parameter value $\delta^*$. However, when $x\in\R^2$, while keeping a one-sided stability region, the convergence is obtained starting from subestimate of $\delta^*$, as reported in the next experiments.

\begin{exm}\label{ex:2d}
Let us consider the classical peridynamic equation of motion \cite{YangOterkusOterkus2024}
\[
\parder{^2\theta}{t^2}(x,y,t)=\frac{6c^2}{\pi\delta^3}\int_0^{2\pi}\int_0^\delta\frac{\theta(x+\xi\cos\varphi,y+\xi\sin\varphi,t)-\theta(x,y,t)}{\xi}\xi\de\xi\de\varphi+f(x,y),
\]
with
\[
f(x,y)\udef-0.05\sin\frac{\pi x}{a}\sin\frac{\pi y}{b},
\]
and initial and boundary conditions given by
\begin{align*}
\theta(-\xi,y) &= -\theta(\xi,y), \\
\theta(a+\xi,y) &= -\theta(a-\xi,y), \\
\theta(x,-\xi) &= -\theta(x,\xi), \\
\theta(x,b+\xi) &= -\theta(x,b-\xi),
\end{align*}
for $\xi\in[0,\delta]$. The exact solution is, in this case,
\[
\theta(x,y,t)=\frac{4}{ab}\frac{1}{c^2}\frac{\pi\delta^3}{6}\sum_{m=1}^\infty\sum_{n=1}^\infty\frac{\left[\int_0^b\int_0^af(x,y)\sin(\overline mx)\sin(\overline ny)\de x\de y\right]\sin(\overline mx)\sin(\overline ny)}{\int_0^{2\pi}\int_0^\delta\frac{1-\cos(\overline m\xi\cos\varphi)\cos(\overline n\xi\sin\varphi)}{\xi}\xi\de\xi\de\varphi},
\]
where $\overline{m}=\frac{\pi m}{a}$ and $\overline{n}=\frac{\pi n}{b}$. \\
Assuming $a=b=1\,\textup{m}$ and $c=1\,\textup{Nm}/\textup{kg}$, we run experiments to learn the value of $\delta>0$, whose true value has been chosen to be $\delta^*=0.1$. \\

For the minimization of $\mathcal{L}$ in \eqref{eq:PINN_lossFunction}, the learning rate for the results shown in Figure \ref{fig:2D_inVal0.095_meanSquared} has been chosen of \texttt{CosineDecay} type, with an initial value of $1e-3$, a decay step of 1000 and warm up step set to zero; the total number of epochs is 1000. \\
In this case, it can be seen that convergence to some value in a small neighborhood of $\delta^*$ is achieved in a monotonically increasing fashion, starting from a subestimate $\delta=0.1-0.005$. \\
Starting from a superestimate $\delta=0.1+0.005$ results in a divergence behavior, as shown in Figure \ref{fig:2D_super_meanSquared}. \\
\begin{figure}[htb]
\centering
\includegraphics[width=0.32\textwidth]{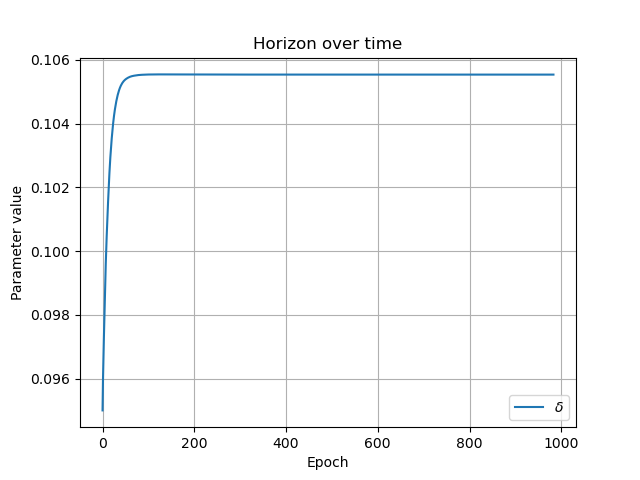}
\includegraphics[width=0.32\textwidth]{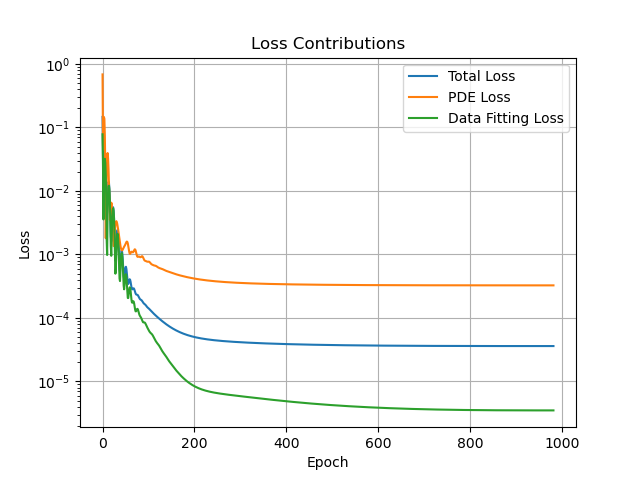}
\includegraphics[width=0.32\textwidth]{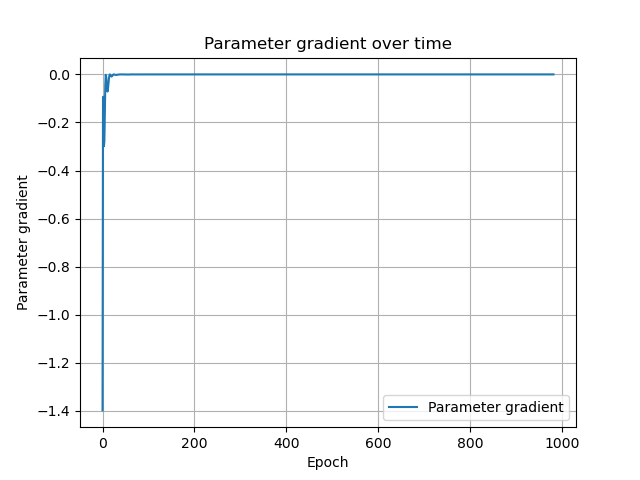}
\caption{Parameter learning for Example \ref{ex:2d} when starting at $\delta=0.095$; the true value for the parameter is $\delta^*=0.1$. The loss function is $\mathcal{L}$ as in \eqref{eq:PINN_lossFunction} and the learning rate follows a \texttt{CosineDecay} scheduler.}
\label{fig:2D_inVal0.095_meanSquared}
\end{figure}
\begin{figure}[htb]
\centering
\includegraphics[width=0.32\textwidth]{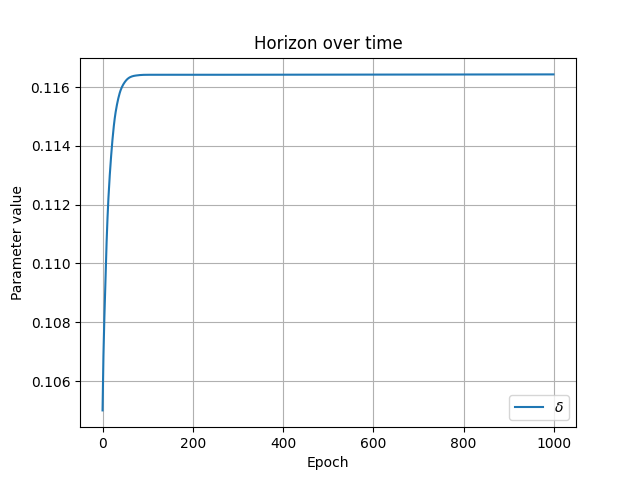}
\includegraphics[width=0.32\textwidth]{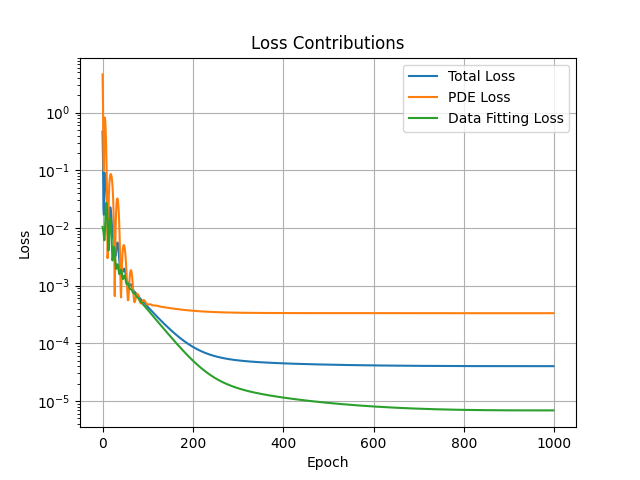}
\includegraphics[width=0.32\textwidth]{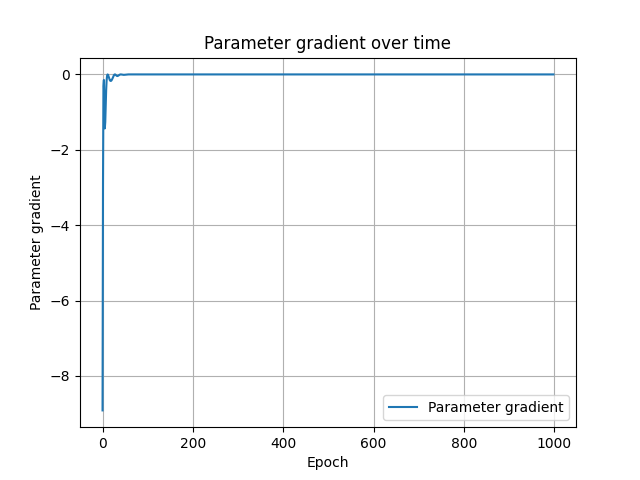}
\caption{Parameter learning for Example \ref{ex:2d} when starting at $\delta=0.105$; the true value for the parameter is $\delta^*=0.1$. The loss function is $\mathcal{L}$ as in \eqref{eq:PINN_lossFunction} and the learning rate follows a \texttt{CosineDecay} scheduler.}
\label{fig:2D_super_meanSquared}
\end{figure}

For the minimization of $\mathcal{L}_2$ in \eqref{eq:PINN_lossFunction_norm}, the learning rate for the results shown in Figure \ref{fig:2D_inVal0.095_norm} has been chosen of \texttt{CosineDecay} type, with an initial value of $1e-4$, a decay step of 1000 and warm up step set to zero; the total number of epochs is 1000. \\
In this case, it can be seen that convergence to some value in a small neighborhood of $\delta^*$ is achieved in a monotonically increasing fashion, starting from a subestimate $\delta=0.1-0.005$; however, we now notice stagnation for both residuals as $\delta^*$ is approached, as reported in Remark \ref{rem:norm}. \\
Starting from a superestimate $\delta=0.1+0.005$ results in a divergence behavior, as shown in Figure \ref{fig:2D_super_norm}.

\begin{figure}[htb]
\centering
\includegraphics[width=0.32\textwidth]{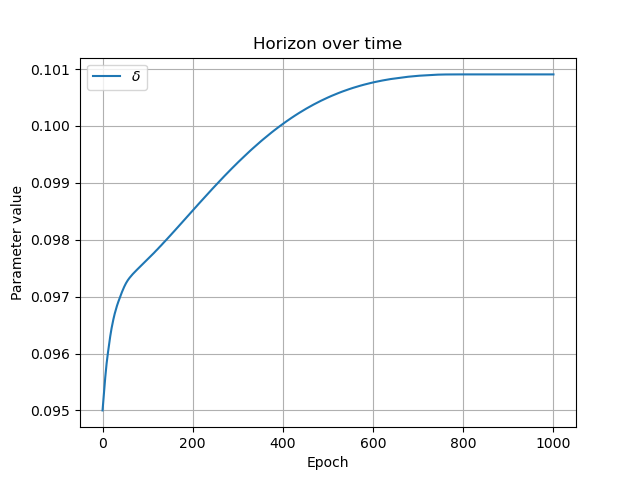}
\includegraphics[width=0.32\textwidth]{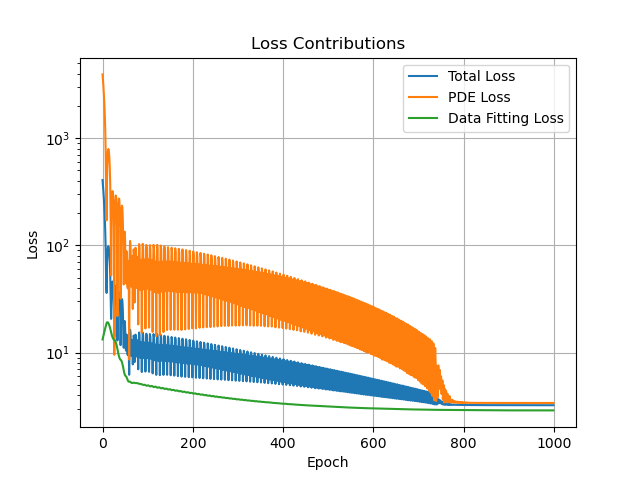}
\includegraphics[width=0.32\textwidth]{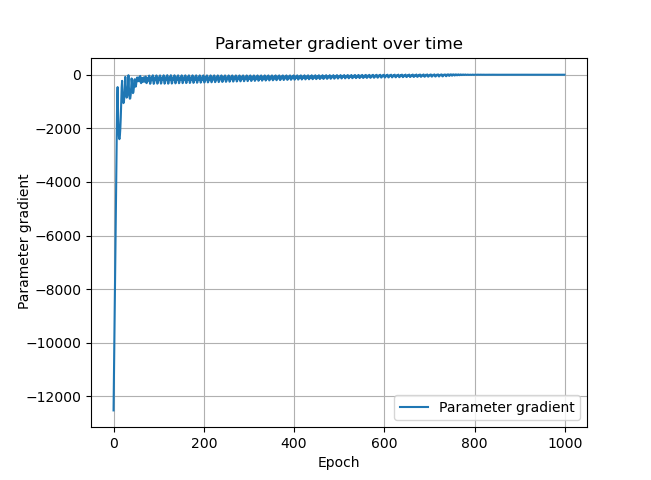}
\caption{Parameter learning for Example \ref{ex:2d} when starting at $\delta=0.095$; the true value for the parameter is $\delta^*=0.1$. The loss function is $\mathcal{L}_2$ as in \eqref{eq:PINN_lossFunction_norm} and the learning rate follows a \texttt{CosineDecay} scheduler.}
\label{fig:2D_inVal0.095_norm}
\end{figure}
\begin{figure}[htb]
\centering
\includegraphics[width=0.32\textwidth]{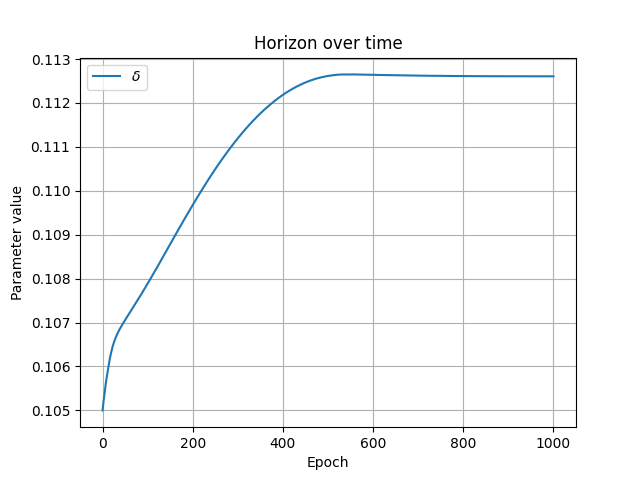}
\includegraphics[width=0.32\textwidth]{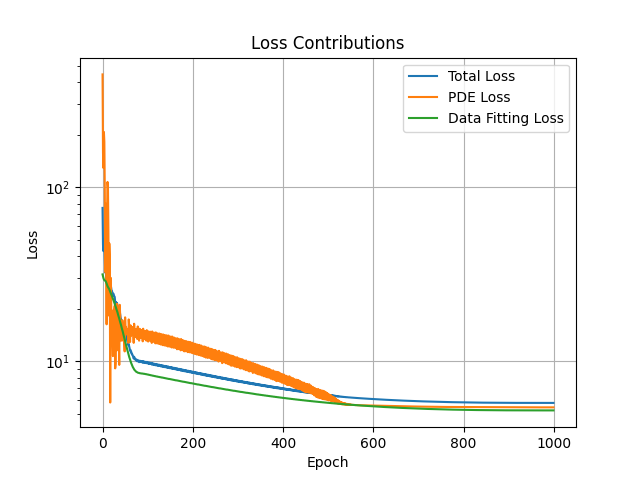}
\includegraphics[width=0.32\textwidth]{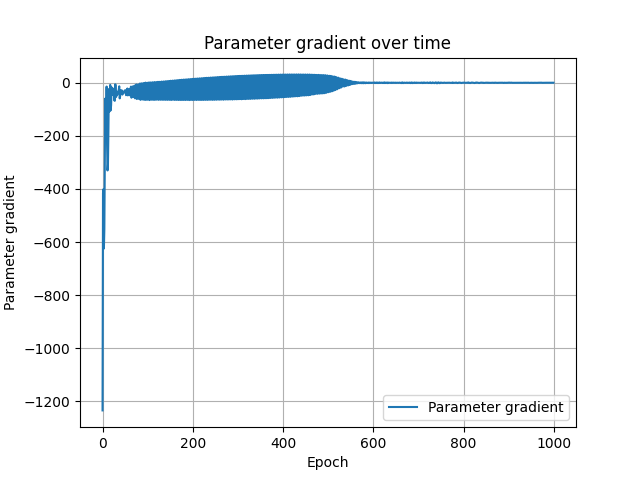}
\caption{Parameter learning for Example \ref{ex:2d} when starting at $\delta=0.105$; the true value for the parameter is $\delta^*=0.1$. The loss function is $\mathcal{L}_2$ as in \eqref{eq:PINN_lossFunction_norm} and the learning rate follows a \texttt{CosineDecay} scheduler.}
\label{fig:2D_super_norm}
\end{figure}
\end{exm}

From Example \ref{ex:2d}, it comes out that convergence occurs only when starting from a subestimate of the true value, while SGD diverges otherwise. This behavior is analogous but opposite than in the 1D case, where SGD provided convergence when starting from a superestimate. It would be a natural outcome to further investigate along this direction, and establish a general pattern for it.

\begin{rem}
It is worth stressing that, as long as the learning rate satisfies the conditions in Proposition \ref{prop:lossFunction} and Proposition \ref{prop:diffRes}, the learning process is expected to converge independently on the specific learning rate chosen for the simulation. In fact, this is what we have seen with our simulations, where different choices of the learning rate have provided graphs with a qualitative comparable behaviors, retaining the same salient properties relative to the convergence of the parameter $\delta$.
\end{rem}

\section{Conclusions}\label{sec:conclusion}

In this work we have tackled the problem to compute the horizon size of the kernel function in bond-based peridynamic 1D and 2D models. We have witnessed that there needs a consistent choice of the initial guess for achieving convergence. In order to explore this phenomenon, stemming from a multi-objective optimization analysis of the PINN loss function, we have first proved that a sufficiently wide neural network, under mild assumptions, is required to attain convergence to a global minimum in a neighborhood of the parameter initialization; then, we provided a result showing that the convergence is indeed monotone, and a bad choice of the initial guess results in a divergence behavior from the exact solution. The proof relies on the assumption that the neural network becomes more and more insensitive to the parameter as it approaches its limit value.

The theoretical results focus on a specific PINN architecture (euclidean loss) and might not hold true for other loss functions or network configurations. Exploring the behavior of PINNs with different learning strategies for event horizon identification is an important area for future research.

Overall, Theorem~\ref{thm:estimate} provides insights into the challenges and limitations of using PINNs to identify the event horizon size in peridynamic models. It highlights the importance of careful parameter initialization and the need for further research to develop more robust and generalizable approaches in this context.

Additionally, in order to perform a qualitative analysis of the PINN architecture with respect to more classical FEM approach, we plan to address the comparisons of these two methods in a future work.

\section*{Acknowledgments}

The three authors gratefully acknowledge the support of INdAM-GNCS 2023 Project, grant number CUP$\_$E53C22001930001, and INdAM-GNCS 2024 project, grant number CUP$\_$E53C23001670001. They are also part of the INdAM research group GNCS.

FVD and LL has been partially funded by PRIN2022PNRR n. P2022M7JZW \emph{SAFER MESH - Sustainable mAnagement oF watEr Resources ModEls and numerical MetHods} research grant, funded by the Italian Ministry of Universities and Research (MUR) and  by the European Union through Next Generation EU, M4C2, CUP H53D23008930001.

SFP has been supported by \textit{PNRR MUR - M4C2} project, grant number N00000013 - CUP D93C22000430001.

The authors want to thank the anonymous reviewers for their comments, that helped to improve the quality of the paper.

\bibliographystyle{plain}
\bibliography{periPINN_rev1.bib}


\end{document}